%% file: 00_main.tex
\documentclass[11pt]{article}

\include{packages}
\include{commands}

\include{environments}

\title{Spanning trees in dense directed graphs}

\author{Amarja Kathapurkar\thanks{University of Birmingham, Birmingham, B15 2TT, UK. Email: {\tt AXK516@student.bham.ac.uk}.}\quad\quad  Richard Montgomery\thanks{University of Birmingham, Birmingham, B15 2TT, UK. Email: {\tt r.h.montgomery@bham.ac.uk}. 
Supported by the European Research Council (ERC) under the European Union Horizon 2020 research and
innovation programme (grant agreement No.\ 947978) and the Leverhulme Trust.
}}
\date{}

\begin{document}

\maketitle

\begin{abstract}
In 2001, Koml\'os, S\'ark\"ozy and Szemer\'edi proved that, for each $\alpha>0$, there is some $c>0$ and $n_0$ such that, if $n\geq n_0$, then every $n$-vertex graph with minimum degree at least $(1/2+\alpha)n$ contains a copy of every $n$-vertex tree with maximum degree at most $cn/\log n$. We prove the corresponding result for directed graphs. That is, for each $\alpha>0$, there is some $c>0$ and $n_0$ such that, if $n\geq n_0$, then every $n$-vertex directed graph with minimum semi-degree at least $(1/2+\alpha)n$ contains a copy of every $n$-vertex oriented tree whose underlying  maximum degree is at most $cn/\log n$.

As with Koml\'os, S\'ark\"ozy and Szemer\'edi's theorem, this is tight up to the value of $c$. Our result improves a recent result of Mycroft and Naia, which requires the oriented trees to have underlying maximum degree at most $\Delta$, for any constant $\Delta\in \mathbb{N}$ and sufficiently large $n$. 
In contrast to these results, our methods do not use Szemer\'edi's regularity lemma.

%
\end{abstract}

\input{01_introduction}

\input{02_preliminaries}
\input{03_almost_spanning_trees}
\input{04_absorption}
\subsection*{Acknowledgements} The authors would like to thank the anonymous referees for their helpful comments that improved this paper.

\bibliographystyle{abbrv}
\bibliography{bibliography}

\end{document}

%% file: packages.tex
\usepackage[left=4cm, right=4cm, top=4cm, bottom=2.25cm]{geometry}
\usepackage{amsmath,color}
\usepackage{amsfonts,amsthm}
\usepackage{amssymb,enumerate,enumitem,verbatim}
\usepackage[pdftex]{graphicx}
\usepackage{hyperref}
\usepackage{setspace,scalefnt}
\usepackage[usenames,dvipsnames,svgnames,table]{xcolor}
\usepackage{tikz,caption,subcaption}
\usepackage[capitalise]{cleveref}

\usepackage{etoolbox} 

%% file: commands.tex
\global\long\def\eps{\varepsilon}   
\global\long\def\N{\mathbb{N}}      

\newcommand{\E}{\mathbb{E}}         
\newcommand{\Prob}{\mathbb{P}}      
\newcommand{\R}{\mathcal{R}}        

\newcommand{\size}[1]{\left\lvert{#1}\right\rvert}      

\newcommand{\lb}{\left(}        
\newcommand{\rb}{\right)}       
\newcommand{\lsb}{\left[}       
\newcommand{\rsb}{\right]}      


%% file: environments.tex
\newtheorem{lemma}{Lemma}[section]
\newtheorem{corollary}[lemma]{Corollary}
\newtheorem{theorem}[lemma]{Theorem}
\newtheorem{proposition}[lemma]{Proposition}
\newtheorem{claim}[lemma]{Claim}
\newtheorem{defn}[lemma]{Definition}

\AtBeginEnvironment{theorem}{\setlist[enumerate]{label=\upshape(\roman*)}}
\AtBeginEnvironment{defn}{\setlist[enumerate]{label=\upshape(\roman*)}}
\AtBeginEnvironment{proposition}{\setlist[enumerate]{label=\upshape(\roman*)}}
\AtBeginEnvironment{lemma}{\setlist[enumerate]{label=\upshape(\roman*)}}
\AtBeginEnvironment{claim}{\setlist[enumerate]{label=\upshape(\roman*)}}
\AtBeginEnvironment{corollary}{\setlist[enumerate]{label=\upshape(\roman*)}}

%% file: 01_introduction.tex
\section{Introduction}\label{sec:intro}
Given two graphs $H$ and $G$, when may we expect to find a copy of $H$ in $G$? In general, this decision problem is NP-complete, and therefore we seek simple conditions on $G$ which imply it contains a copy of $H$. An important early result is Dirac's theorem from 1952 that, when $n\geq 3$, any $n$-vertex graph with minimum degree at least $n/2$ contains a cycle through every vertex, that is, a Hamilton cycle. This is a particular instance of the following meta-question, which has seen much subsequent study. Given an $n$-vertex graph $H$, what is the lowest minimum degree condition on an $n$-vertex graph $G$ which guarantees it contains a copy of $H$? As such a copy of $H$ would contain every vertex in $G$, we say it is a \emph{spanning} copy of $H$.

This question has been studied for many different graphs $H$, for example when $H$ is a $K$-factor for some small fixed graph $K$~\cite{HajSzem,DDminfactor}, the $k$-th power of a Hamilton cycle for any $k\geq 2$~\cite{KSSpowercycle} and when $H$ has bounded chromatic number and maximum degree, and sublinear bandwith~\cite{bandwidth}. For more details on these results, and those for other graphs, see the survey by K\"uhn and Osthus~\cite{DDgraphsurvey}. Here, we will concentrate on the minimum degree required to guarantee different spanning trees.  

Koml\'os, S\'ark\"ozy and Szemer\'edi~\cite{KSS95} proved in 1995 that, for each $\alpha,\Delta>0$, there is some $n_0$ such that, if $n\geq n_0$, then every $n$-vertex graph with minimum degree at least $(1/2+\alpha)n$ contains a copy of every $n$-vertex tree with maximum degree at most $\Delta$, thus confirming a conjecture of Bollob\'as~\cite{belabook}. This result is furthermore notable as one of the earliest applications of the blow-up lemma. In 2001, Koml\'os, S\'ark\"ozy and Szemer\'edi~\cite{KSS01} relaxed the maximum degree condition, showing that, for each $\alpha>0$, there is some $c>0$ and $n_0$ such that, if $n\geq n_0$, then every $n$-vertex graph with minimum degree at least $(1/2+\alpha)n$ contains a copy of every $n$-vertex tree with maximum degree at most $cn/\log n$. This is tight up to the constant $c$. In 2010, Csaba, Levitt, Nagy-Gy\"orgy and Szemer\'edi~\cite{csaba2010tight} showed that, in the other direction, the degree bound in the graph can be reduced for trees with constant maximum degree. That is, they showed that, for each $\Delta>0$, there is some $C>0$ such that   every $n$-vertex graph with minimum degree at least $n/2+C\log n$ contains a copy of every $n$-vertex tree with maximum degree at most $\Delta$. This is tight up the constant $C$, and, moreover, unlike the previous results, did not use Szemer\'edi's regularity lemma.

In this paper, we will prove the corresponding version of this result for \emph{directed graphs (digraphs)}.

The \emph{minimum semidegree} of a digraph $D$, denoted by $\delta^0(D)$, is the smallest in- or out-degree over the vertices in $D$, that is, $\delta^0(D)=\min_{v\in V(D),\diamond\in \{+,-\}}d^\diamond(v)$.
Ghouila-Houri~\cite{GH} solved the minimum semidegree problem for the directed Hamilton cycle, showing that, if an $n$-vertex digraph $D$ has $\delta^0(D)\geq n/2$, then it contains a directed Hamilton cycle. That is, an $n$-vertex cycle with the edges oriented in the same direction.
DeBiasio, K\"uhn, Molla, Osthus and Taylor~\cite{anyori} showed that, when $n$ is sufficiently large, this holds in fact for any $n$-vertex cycle with any orientations on its edges, except for when the edges change direction at every vertex around the cycle. This latter cycle, known as the \emph{anti-directed Hamilton cycle}, is only guaranteed to appear if $\delta^0(D)\geq n/2+1$, as shown by DeBiasio and Molla~\cite{DeMol}.

Recently, Mycroft and Naia~\cite{RichardTassio,mycroft2020trees} gave the first bound on the minimum semidegree required for the appearance of different spanning trees. Here, $H$ is an oriented $n$-vertex tree, with some bound on the degree of its underlying (undirected) tree. Mycroft and Naia~\cite{RichardTassio,mycroft2020trees} proved that, for each $\alpha,\Delta>0$, there is some $n_0$ such that, if $n\geq n_0$, then every $n$-vertex digraph with minimum semidegree at least $(1/2+\alpha)n$ contains a copy of every oriented $n$-vertex tree $T$ with $\Delta^\pm(T)\leq \Delta$. Moreover, their result holds for a slightly wider class of trees, allowing them to show that, for each $\alpha>0$, almost every labelled oriented $n$-vertex tree appears in every $n$-vertex digraph with minimum semidegree at least $(1/2+\alpha)n$.
  
In this paper, we introduce new methods to embed oriented trees in digraphs, relaxing the maximum degree condition to give a full directed version of Koml\'os, S\'ark\"ozy and Szemer\'edi's result, as follows.

\begin{theorem}\label{full}
For each $\alpha>0$, there exists $c>0$ and $n_0\in \mathbb{N}$ such that the following holds for every $n \geq n_0$. Any $n$-vertex digraph $D$ with $\delta^0(D)\geq (1/2+\alpha)n$ contains a copy of every oriented $n$-vertex tree $T$ with $\Delta^\pm(T)\leq cn/ \log n$.
\end{theorem}

We note that the undirected version follows immediately from Theorem~\ref{full}. Indeed, given any $n$-vertex tree $T$ and an $n$-vertex graph $G$, we can apply Theorem~\ref{full} to a copy of $T$ with each edge oriented arbitrarily and a digraph formed from $G$ by replacing each edge $uv$ with an edge from $u$ to $v$ and an edge from $v$ to $u$. This demonstrates that, as with Koml\'os, S\'ark\"ozy and Szemer\'edi's result, Theorem~\ref{full} is tight up to the constant $c$. Furthermore, through Theorem~\ref{full} we give a new proof of the undirected result without using Szemer\'edi's regularity lemma, in contrast to the work of both Koml\'os, S\'ark\"ozy and Szemer\'edi~\cite{KSS95}, and Mycroft and Naia~\cite{RichardTassio,mycroft2020trees}, adding to the non-regularity proof for trees with constant maximum degree by Csaba, Levitt, Nagy-Gy\"orgy and Szemer\'edi~\cite{csaba2010tight} described above. Key to our result is to use a random embedding of part of the tree using `guide sets' and embedding many leaves (and small subtrees) of the tree using `guide graphs'. This replaces the regularity methods of~\cite{KSS95,RichardTassio,mycroft2020trees}, and is sketched in Section~\ref{sec:prelim}, where we also outline the rest of this paper.

%% file: 02_preliminaries.tex
\section{Preliminaries}\label{sec:prelim}
\subsection{Notation}

Let $D$ be a digraph. We denote by $V(D)$ and $E(D)$ the vertex set and edge set of $D$, respectively, where every element of the edge set of $D$ is an ordered pair of vertices. We let $\size{D} = \size{V(D)}$, which we call the \emph{size} of $D$, and let $e(D) = \size{E(D)}$. Letting $u,v \in V(D)$, if $uv \in E(D)$, then we say that $u$ is an \emph{in-neighbour} of $v$ and $v$ is an \emph{out-neighbour} of $u$.  Denote by $N^-_D(v)$ and $N_D^+(v)$, respectively, the set of all in- and out-neighbours of $v$. We let $d^-_D(v) =  \size{N_D^-(v)}$ and $d_D^+(v) = \size{N_D^+(v)}$, and we refer to these as the \emph{in-} and \emph{out-degree of $v$}, respectively. For each $\diamond \in \{+, -\}$, we let $\delta^\diamond(D)$ and $\Delta^\diamond(D)$ be, respectively, the \emph{minimum} and \emph{maximum $\diamond$-degree} of $D$. For any $A,B \subseteq V(D)$, and each $\diamond \in \{+, - \}$, let $N_D^\diamond(A,B)=\bigcup_{a \in A} (N^\diamond_D(a) \cap B)$, and let $d_D^\diamond(A,B)=\size{N_D^\diamond(A,B)}$. We omit the subscript when the graph is clear from context.  Note that, for simplicity of notation, we use `$-$' and `in' interchangeably, and, similarly, we use `$+$' and `out' interchangeably. We use `$\pm$' to represent that a property holds for both `$-$' and `$+$'. 

Suppose that $A$ and $B$ are disjoint subsets of $V(D)$. We write $D[A]$ to mean $D$ induced on the set $A$, that is, the graph obtained from $D$ by deleting all vertices which are not in $A$. For each $\diamond \in \{+, -\}$, a \emph{$\diamond$-matching} from $A$ into $B$ is a set of vertex-disjoint edges such that every edge in the set has one endpoint in $A$ and one endpoint in $B$, and the endpoint in $B$ is a $\diamond$-neighbour of the endpoint in $A$, that is, every edge is a $\diamond$-edge from $A$ into $B$. We say this matching \emph{covers $A$} if every vertex of $A$ belongs to some edge in the matching, and we call this a \emph{perfect $\diamond$-matching} if it covers both $A$ and $B$. A \emph{bare path} of length $m$ in a tree is a path with $m$ edges such that each of the internal vertices have degree 2 in the tree. When $P$ is a path in $D$, we let $D-P$ denote the subgraph of $D$ obtained by removing the internal vertices of $P$. 

For any $n \in \N$, we let $[n] := \{1, \dots, n\}$. In order to simplify notation, we use hierarchies to state our results. That is, for $a, b \in (0, 1]$, whenever we write that a statement holds for $a \ll b$ (or $b \gg a$), we mean that there exists a non-decreasing function $f\colon (0,1] \rightarrow (0,1]$ such that the statement holds whenever $a \leq f(b)$. We define similar expressions with multiple variables analogously. We say a random event occurs \emph{with high probability} if the probability of the event occurring tends to 1 as $n$ tends to infinity. In our proofs, when we have shown that a property holds with high probability, we will implicitly assume that this property holds from that point onwards. For simplicity, we ignore floors and ceilings wherever this does not affect the argument.

\subsection{Proof sketch}\label{subsec:sketch}
When $1/n\ll c\ll \alpha$, we will embed any oriented $n$-vertex tree $T$ with $\Delta^\pm(T)\leq cn/\log n$ into any $n$-vertex digraph $D$ with $\delta^0(D)\geq (1/2+\alpha)n$.
We embed $T$ using the absorption method, an approach first introduced in general by R\"odl, Ruci\'nski and Szemer\'edi~\cite{RRSab} which has been effective on a range of embedding problems for spanning graphs and digraphs (see, for example, the survey~\cite{bottsurvey}). We first partially embed a subtree $T''$ of $T$ into a set $A$ such that, given any subset $B\subset V(D)$ with $A\subset B$ and $|B|=|T''|$, we can complete this embedding of $T''$ into $D[B]$ (see Theorem~\ref{switching}).

We then use an almost-spanning embedding to embed the vertices in $V(T)\setminus V(T'')$ to extend the partial embedding of $T''$ (see Theorem~\ref{almost}). We will have chosen $T''$ so that in this stage a tree, called $T'$, is attached to an embedded vertex of $T''$. Using the property of the partial embedding of $T''$, we then complete the embedding of $T''$ with the unused vertices in $D$. The decomposition of $T$ that we need follows from a simple proposition (Proposition~\ref{littletree}).

In Section~\ref{proofoffull}, we state these three results, Theorem~\ref{switching}, Theorem~\ref{almost} and Proposition~\ref{littletree}, before deducing Theorem~\ref{full} from them. In Section~\ref{sec:sk1}, we then discuss in detail the proof of Theorem~\ref{almost}, which is the major challenge overcome by this paper.

In the rest of Section~\ref{sec:prelim}, we restate the probabilistic tools we will use, and give a basic structural decomposition of trees and some simple results on matchings. In Section~\ref{sec:almost}, we prove Theorem~\ref{almost}. In Section~\ref{sec:switching}, we prove Theorem~\ref{switching}.

\subsubsection{Main tools and deduction of Theorem~\ref{full}}\label{proofoffull}
For Theorem~\ref{full}, we will first find a suitable subtree $T''\subset T$ and a set $A\subset V(D)$ with slightly fewer than $|T''|$ vertices, so that, given any set $B$ of $|T''|$ vertices containing $A$, we can embed $T''$ in $D[B]$. Furthermore, we will ensure that some pre-specified vertex $t\in V(T'')$ is always embedded to some fixed vertex $v\in A$, as follows.

\begin{theorem}\label{switching}
Let $1/n\ll c \ll  \eps\ll \mu \ll \alpha$. Let $D$ be an $n$-vertex digraph with minimum semidegree at least $(1/2+\alpha)n$. Let $T$ be an oriented tree with $\mu n$ vertices and $\Delta^\pm(T)\leq cn/\log n$, and let $t\in V(T)$.

Then, $V(D)$ contains a vertex set $A$ with size $(\mu-\eps)n$ containing a vertex $v\in A$ such that the following holds. For any set $B\subset V(D)$ with $A\subset B$ and $|B|=\mu n$, $D[B]$ contains a copy of $T$ in which $t$ is copied to $v$.
\end{theorem}

Theorem~\ref{switching} is proved in Section~\ref{sec:switching} by randomly embedding most of $T$ and taking $A$ to be the image of this embedding. We then show that the partial embedding of $T$ can be extended using any new vertex in $y\in V(D)\setminus A$ by switching $y$ into the partial embedding in place of some vertex in $A$ that can instead be used to embed a new vertex of $T$. Repeatedly doing this will allow the embedding of $T$ to be completed using any set of $|T|-|A|$ new vertices in $V(D)\setminus A$. This is sketched in more detail at the start of Section~\ref{sec:switching}, before Theorem~\ref{switching} is proved.

We will embed the majority of the tree for \cref{full}, using the following almost-spanning embedding.

\begin{theorem}\label{almost}
Let $1/n\ll c\ll\eps,\alpha$. Let $D$ be an $n$-vertex digraph with minimum semidegree at least $(1/2+\alpha)n$ and let $v\in V(D)$. Let $T$ be an oriented tree with at most $(1-\eps)n$ vertices and $\Delta^\pm(T)\leq c n/\log n$, and let $t\in V(T)$.

Then, $D$ contains a copy of $T$ in which $t$ is copied to $v$.
\end{theorem}

Using in addition the following simple proposition (see, for example,~\cite[Proposition 3.22]{montgomery2019}), we can now deduce \cref{full}.

\begin{proposition}\label{littletree} Let $n,m\in \N$ satisfy $1\leq m\leq n/3$. Given any $n$-vertex tree~$T$, there are two edge-disjoint trees $T_1,T_2\subset T$ such that $E(T_1)\cup E(T_2)=E(T)$ and $m\leq |T_2|\leq 3m$.
\end{proposition}

\begin{proof}[Proof of \cref{full} from \cref{switching,almost}]
Let $\eps,\mu>0$ be such that $c \ll \eps \ll \mu \ll \alpha$. Let $D$ be an $n$-vertex digraph with $\delta^0(D)\geq (1/2+\alpha)n$. Let $T$ be an oriented $n$-vertex tree with $\Delta^\pm(T)\leq cn/\log n$.

Using \cref{littletree} with $m = \mu n$, find edge-disjoint trees $T',T''\subset T$ such that $E(T')\cup E(T'')=E(T)$ and $\mu n\leq |T''|\leq 3\mu n$. Let $t$ be the vertex which is in both $T'$ and $T''$.
By \cref{switching} applied with $\mu'=|T''|/n$, there is a set $A\subset V(D)$ such that $\size{A} = |T''|- \eps n$, and a vertex $v \in A$ such that, for any set $B \subset V(D)$ with $A \subset B$ and $\size{B} = |T''|$, $D[B]$ contains a copy of $T''$ in which $t$ is copied to $v$.

Let $D' = D-(A \setminus \{v\})$. Let $n'=\size{D'}$, so that $(1-3\mu) n\leq n'\leq n$. Let $\alpha'$ be such that $D'$ has minimum semidegree $(1/2+\alpha')n'$. Note that $(1/2+\alpha')n\geq (1/2+\alpha')n'\geq (1/2 + \alpha-3\mu) n$, so that $\alpha'\geq \alpha/2$.
Furthermore, $n'=n-|T''|+\eps n+1= |T'|+\eps n$, and therefore
\[
\frac{|T'|}{n'}= \frac{|T'|}{|T'|+\eps n}\leq \frac{|T'|}{|T'|(1+\eps)}\leq 1-\eps/2.
\]
Thus, by \cref{almost}, we can find a copy, $S'$ say, of $T'$ in $D'$ in which $t$ is copied to $v$. By applying the property of $A$ from \cref{switching}, we can then find a copy of $T''$ in $D-(V(S')\setminus \{v\})$ in which $t$ is copied to $v$. Together, these give us a copy of $T$.
\end{proof}

\subsubsection{Proof Sketch of Theorem~\ref{almost}}\label{sec:sk1}
We will embed a $(1-\eps)n$-vertex tree $T$ for Theorem~\ref{almost} by dividing most of $T$ into a small core forest $T_0\subset T$ and a collection of constant-sized subtrees, which are either attached to $T_0$ by a single edge or by two short paths. It is the trees attached to $T_0$ by a single edge that will be the most challenging to embed, and so we dedicate most of our attention in the proof sketch to this.

More precisely, we will find a tree $T'\subset T$, containing a \emph{core forest} $T_0\subset T'$ and vertex-disjoint trees $S_1,\ldots,S_\ell\subset T'-V(T_0)$, for some $\ell\in \mathbb{N}$, such that $T'$ is formed from $T_0$ by, for each $i\in [\ell]$,
\begin{enumerate}[label = (\arabic{enumi})]   
    \item \label{item:case_1}either adding $S_i$ to $T_0$ using two bare paths with length 2,
    \item \label{item:case_2} or adding $S_i$ to $T_0$ with a single edge.
\end{enumerate}
Furthermore, for some $\mu>0$ and $K\in \N$, with  $1/n\ll 1/K,\mu\ll \alpha,\eps$, we will have that
\begin{itemize}
    \item $|T_0|\leq \mu n$ (i.e., $T_0$ is small),
    \item $|T'|\geq |T|-\mu n$ (i.e., $T'$ is most of $T$),
    \item there are at most $\mu n$ trees $S_i$ which are in Case (1), and
    \item each tree $S_i$ has at most $K$ vertices.
\end{itemize}

In Case \ref{item:case_1}, we say $S_i$ is added to $T_0$ as a path, and in Case \ref{item:case_2} we say $S_i$ is added to $T_0$ as a leaf. 
The crux of our method is to embed $T_0$ along with the trees $S_i$ in Case \ref{item:case_2} connected to the embedding of $T_0$ by the appropriate edge. This is encapsulated in the following lemma, which is proved in Section~\ref{sec:csstars}.

\begin{lemma}\label{lem:CSstars} Let $1/n\ll c\ll \mu\ll \alpha,\eps$, let $c\ll 1/K$ and let $\ell\in \mathbb{N}$.  Suppose $D$ is an $n$-vertex digraph with $\delta^0(D)\geq (1/2+\alpha)n$ and $v \in V(D)$.

Suppose that $T$ is an oriented tree with $|T|\leq (1-\eps)n$ and $\Delta^\pm(T)\leq cn/\log n$. Suppose that $T',S_1,\ldots S_\ell \subset T$ are vertex-disjoint subtrees with $|T'|\leq \mu n$, and $|S_i|\leq K$ for each $i\in [\ell]$. Suppose that $T$ is formed from $T'$ by attaching each $S_i$, $i\in [\ell]$, to $T'$ by an edge. Finally, let $t\in V(T')$.

Then, $D$ contains a copy of $T$ in which $t$ is copied to $v$.
\end{lemma}

We will now briefly sketch how Theorem~\ref{almost} can be proved from Lemma~\ref{lem:CSstars}. Let $m$ be the total number of vertices that appear in the trees $S_i$ in Case (1) above. To embed these trees, we use the fact that two random sets in $D$ of the same (linear) size are likely to have a perfect matching from one to the other (see Proposition~\ref{prop:matching_between_sets}). Taking $p\gg 1/n$ and $Kp\leq 1$, we can, with high probability, find $pn$ copies of an oriented tree with $K$ vertices in a random set of $Kpn$ vertices in $D$ by taking randomly $K$ disjoint subsets within this set of size $pn$ and finding appropriate matchings between them (see Section~\ref{sec:match}). Collecting isomorphic trees $S_i$ together, and applying this to each of the constantly many (depending on $K$) isomorphism classes, allows us to embed the trees $S_i$ in Case (1) with high probability in a random set with size $m+\eps n/4$. Here, the extra $\eps n/4$ vertices allow us to find a linear number of trees in each isomorphism class by finding some additional trees if required.

Thus, in a partition of $V(D)$ into sets $V_1\cup V_2\cup V_3\cup V_4$ chosen uniformly at random so that $|V_1|=n-m-3\eps n/4$, $|V_2|=m+\eps n/4$, $|V_3|=|V_4|=\eps n/4$, with high probability, the following occur.
\begin{itemize}
\item $\delta^{\pm}(D[V_1])\geq (1/2+\alpha/2)|V_1|$, so that, applying Lemma~\ref{lem:CSstars}, we can embed $T_0$ along with the trees $S_i$ in Case \ref{item:case_2} connected to the embedding of $T_0$ by the appropriate edge. 
\item We can embed the trees $S_i$ in Case \ref{item:case_1} in $D[V_2]$.
\item  Then, using that there are at most $\mu n$ trees in Case \ref{item:case_1}, we can greedily attach them to the embedding of $T_0$ using two paths with length 2 whose interior vertex is an unused vertex in $V_3$ (see Section~\ref{sec:CSaspaths}).
\item Finally, as $|T|-|T'|\leq \mu n$, we can greedily extend the resulting embedding of $T'$ to one of $T$, by adding a sequence of leaves using vertices in $V_4$ (see Section~\ref{sec:proofofalmost}).
    \end{itemize}
Here, the last two steps are (with high probability) possible using the semi-degree condition of $D$. Note that, as $\mu\ll \eps$, we only embed a small proportion of vertices into $V_3$ and $V_4$.

We will now give a detailed proof sketch of Lemma~\ref{lem:CSstars}.

\subsubsection*{Proof sketch of Lemma~\ref{lem:CSstars}}
To simplify our discussion, let us assume that each tree $S_i$ in Lemma~\ref{lem:CSstars} consists of only a single vertex, which is an out-neighbour in the tree $T$ of a vertex of $T_0$, and that every vertex in $T_0$ is attached to exactly one such tree. That is, $T$ consists of $T_0$ with an out-matching attached. 
Our embedding of $T_0$ is randomised, which will allow the methods described to be used to find matchings attached from different subsets of the image of the embedding of $T_0$ to different random sets. This will allow the embedding below for $T_0$ to be used for the general case.

Let us detail the example situation precisely. Suppose we have a $\mu n$-vertex tree $T_0$ and choose two disjoint random sets $V_0,V_1\subset V(D)$ with size $p_0n$ and $p_1n$ respectively, where $p_0\gg \mu$ and $p_1=(1+o(1))\mu$. We will randomly embed $T_0$ into $V_0$, so that there is an out-matching from the vertex set of the embedding of $T_0$ into $V_1$. 
More generally, we may have to attach matchings into several different sets from $T_0$, so we can only use a small proportion of spare vertices in $D$ (and so $p_1$ is only a little larger than $\mu$). On the other hand, as these matchings will be all attached to the same small tree, $T_0$, we can use many spare vertices when embedding $T_0$ (and so we take $p_0\gg \mu$).



We will embed $T_0$ vertex-by-vertex, say in order $t_1,\ldots,t_\ell$, so that each new vertex is embedded as an in- or out-leaf of the previously embedded subtree. Having chosen the random sets $V_0,V_1$, and before beginning the embedding, we will find \emph{guide sets} $A_{v,\diamond}\subset  N^\diamond_D(v,V_0)$, $v\in V_0$ and $\diamond\in \{+,-\}$, which we use to guide the random embedding. We then start the random embedding, under the rule that if, for some $v\in V_0$ and $\diamond\in\{+,-\}$, we are attaching a $\diamond$-edge as a leaf to $v$, then we choose this leaf uniformly at random from the unused vertices in $A_{v,\diamond}$.

The guide sets ensure that, with high probability, there will be a matching from the embedding of $T_0$ into $V_1$. These guide sets are found using Lemma~\ref{lem-skew-2}, and they exist (with high probability for the choice of $V_0,V_1$) due to the semi-degree condition in $D$.
Essentially, for some constants $\beta,\gamma$, we find, for each $v\in V(D)$ and $\diamond\in \{+,-\}$, a set $A_{v,\diamond}\subset N^\diamond_D(v, V_0)$ with size $\beta n$ and bipartite digraphs $H^\circ_{v,\diamond}\subset D^\circ[A_{v,\diamond},V_1]$, $\circ\in \{+,-\}$, so that in $H^\circ_{v,\diamond}$ each vertex in $A_{v,\diamond}$ has around $\gamma p_1n$ $\circ$-neighbours in $V_1$, and each vertex in $V_1$ has around $\gamma \beta n$ $\circ$-edges leading into it. That is,  $H^\circ_{v,\diamond}$ is approximately regular on each side with edge density approximately $\gamma$.

We use the guide graphs $H^+_{v,\diamond}$ to find the matching from the embedding of $T_0$ by constructing an auxiliary bipartite digraph $K$ with vertex classes $\{s_1,\ldots,s_\ell\}$ and $V_\ell$. In this example situation, $K$ is a subgraph of $D$, and a matching in $K$ corresponds exactly to a matching from the image of $V(T_0)$ to $V_1$. In the more general case we attach multiple different matchings simultaneously and $K$ has vertices from the image of $V(T_0)$ copied different numbers of times (see Section~\ref{sec:almost}). The digraph $K$ does not contain  all the edges in $D$ from $\{s_1,\ldots,s_\ell\}$ to $V_\ell$. Instead, we add edges using the guide graphs so that when we have constructed $K$ it will be, with high probability, approximately regular, so that we can find our required matching via Hall's matching criterion. 

When a vertex $t_i$ is embedded using a guide set $A_{v_i,\diamond_i}$, to some vertex $s_i$ say, we add only the edges in $H^+_{v_i,\diamond_i}$ adjacent to $s_i$ to $K$ -- note that approximately $\gamma p_1n$ edges are added next to $s_i$. Note further that, as most of the vertices in $A_{v_i,\diamond_i}$ will be unused, each $w\in V_1$ will have an edge added from $s_i$ to $w$ with probability approximately
\begin{equation}\label{eq:final}
\frac{d_{H^+_{v_i,\diamond_i}}^-(w)}{|A_{v_i,\diamond_i}|}
\approx
\frac{ \gamma \beta n}{\beta n}=\gamma.
\end{equation}
When this is complete, $K$ is a bipartite digraph with vertex classes $\{s_1,\ldots,s_\ell\}$ and $V_1$. Each vertex $s_i$ will have out-degree approximately $\gamma p_1n$, and, due to the randomness of the embedding and \eqref{eq:final}, each vertex in $V_1$ will have in-degree which is approximately $\gamma \ell=\gamma|T_0|\approx \gamma p_1n$.

Thus, $K$ will be a bipartite graph with the in-degrees in one vertex class approximately equal to the out-degrees in the other. Via Hall's matching criterion, an out-matching will exist from $\{s_1,\ldots,s_\ell\}$ to $V_1$ which covers most of the vertices in $\{s_1,\ldots,s_\ell\}$. By ensuring that $V_1$ is likely to be a little larger than $\ell$, we in fact will get with high probability that such an out-matching can cover $\{s_1,\ldots,s_\ell\}$.

Note that, in the sketch above, we do not use the graph $H_{v,\diamond}^-$. However, in practice, we find such guide sets and guide graphs with $V_1=V(D)\setminus V_0$ (see Lemma~\ref{lem-skew-1}), before taking random subsets of $V_1$. We will find out-matchings into some of these random sets, and in-matchings into some others. Therefore, it is important to have both guide graphs $H^-_{v,\diamond}$ and $H^+_{v,\diamond}$, and, furthermore, that the same set $A_{v,\diamond}$ is used for both graphs.

Finally, let us note where the condition $\Delta^\pm(T)\leq cn/\log n$ is used in our proof of Lemma~\ref{lem:CSstars}. In the sketch above the set $V_1$ will always have size which is linear in $n$, but we may need to attach the trees in Lemma~\ref{lem:CSstars} to few vertices in $T$. The maximum in- or out-degree condition on $T$ ensures, that, if the trees $S_i$ in Lemma~\ref{lem:CSstars} together comprise linearly (in $n$) many vertices in $T$, then they are attached to at least $C\log n$ different vertices, for some large constant $C$, which gives us sufficient probability concentrations when these vertices are randomly embedded for the corresponding versions of Hall's criterion to hold (see the proof of Claim~\ref{claimembed}).

\subsection{Probabilistic tools}
Let $n, m, k \in \N$ be such that $\max\{m, k\} \leq n$. Let $A$ be a set of size $n$, and $B \subseteq A$ be such that $\size{B} = m$. Let $A'$ be a uniformly random subset of $A$ of size $k$. Then the random variable $X = \size{A' \cap B}$ is said to have hypergeometric distribution with parameters $n,m$ and $k$, which we denote by $X \sim \textup{Hyp}(n,m,k)$. We will use the following Chernoff-type bound.
\begin{lemma}[see, for example,  \cite{janson2011random}]\label{lem:chernoff}
Suppose $X \sim \textup{Hyp}(n,m,k)$. Then for any $0 < \alpha < 3/2$, we have
\[
    \Prob\lsb \size{ X - \E[X] } \geq \alpha \E[X] \rsb \leq 2 \exp\lb - \alpha^2 \E[X]/ 3 \rb.
\]
\end{lemma}

 A sequence of random variables  $(X_i)_{i \geq 0}$ is a martingale if $\E[X_{i+1} \mid X_0, \dots, X_i] = X_i$ for each $i\geq 0$. We will use the following Azuma-type bound for martingales.

\begin{lemma}[see, for example, \cite{alon2004probabilistic}]\label{lem:azuma}
Let $(X_i)_{i \geq 0}$ be a martingale and let $c_i>0$ for each $i\geq 1$. If $\size{X_i -X_{i-1}} <c_i$ for each $i\geq 1$, then, for each $n\geq 1$,
\[
\Prob[\size{X_n-X_0} \geq t ] \leq 2 \exp \lb -\frac{t^2}{\sum_{i=1}^nc_i^2} \rb.
\]
\end{lemma}

We will use this bound for supermartingales and submartingales. A sequence of random variables  $(X_i)_{i \geq 0}$ is a supermartingale if $\E[X_{i+1} \mid X_0, \dots, X_i] \leq X_i$ for each $i\geq 0$, and a submartingale if $\E[X_{i+1} \mid X_0, \dots, X_i] \geq X_i$ for each $i\geq 0$. The bound on the upper tail in Lemma~\ref{lem:azuma} holds for supermartingales, while the bound on the lower tail holds for submartingales. We will always use this to bound the sum of random variables using the following simple corollary.

\begin{corollary}\label{cor:azuma}
Let $(Z_i)_{i=1}^n$ be a sequence of random variables. For each $i\in [n]$, let $a_i,c_i\in \mathbb{R}$ be constant such that $|Z_i-a_i|\leq c_i$.
\begin{enumerate}
\item\label{cor:azuma_sup} If $\E[Z_i \mid  Z_1,\ldots,Z_{i-1}]\leq a_i$, then for each $t>0$,
    \[
\Prob\left[\sum_{i=1}^nZ_i\geq \sum_{i=1}^na_i+t\right]\leq 2\exp\left(-\frac{t^2}{\sum_{i=1}^nc_i^2}\right).
\]
    \item\label{cor:azuma_sub} If $\E[Z_i \mid  Z_1,\ldots,Z_{i-1}]\geq a_i$, then for each $t>0$,
    \[
\Prob\left[\sum_{i=1}^nZ_i\leq \sum_{i=1}^na_i-t\right]\leq 2\exp\left(-\frac{t^2}{\sum_{i=1}^nc_i^2}\right).
\]

\end{enumerate}
\end{corollary}

\begin{proof}
We prove \ref{cor:azuma_sup}, and note that \ref{cor:azuma_sub} follows by applying \ref{cor:azuma_sup} to the sequence $(-Z_i)^n_{i=1}$. Let $Y_0 = 0$ and, for each $i \in [n]$, let $Y_i = \sum_{i'\in [i]}Z_{i'}-a_{i'}$. Then, $\E[Y_{i+1} \mid Y_1, \dots, Y_i] = \E[Z_{i+1} - a_{i+1} + Y_i \mid Y_1, \dots, Y_i] \leq Y_i$. Furthermore, for each $i\in [n]$, $\size{Y_{i}-Y_{i-1}} =\size{Z_{i}-a_i} \leq c_i$. Therefore, we can apply \cref{lem:azuma} for supermartingales to show that
 \[
\Prob\left[\sum_{i=1}^nZ_i\geq \sum_{i=1}^na_i+t\right] = \Prob\lsb Y_i - Y_0 \geq t \rsb\leq 2\exp\left(-\frac{t^2}{\sum_{i=1}^nc_i^2}\right).\qedhere
\]
\end{proof}

\subsection{Structural lemmas}
In this section we decompose undirected trees. Note that we will later apply this to directed trees as the edge directions do not affect the decompositions. We will use the following simple but useful lemma (see \cite[Lemma 4.1]{montgomery2018embedding}) which tells us that either a tree has many leaves, or it has many bare paths. 

\begin{lemma}\label{lem:few_leaves_many_bare_paths}
Let $t, s \geq 2$, and suppose that $T$ is a tree with at most $t$ leaves. Then there is some $m$ and some vertex-disjoint bare paths $P_i$, $i \in [m]$, in $T$ with length $s$ so that $\size{T-P_1 - \dots - P_m} \leq 6st + 2\size{T}/(s+1)$.
\end{lemma}

We can now prove the following key lemma, in which we decompose a tree for our embedding.

\begin{lemma}\label{lem:decomposing_trees}
Let $0\ll 1/n\ll 1/K\ll 1/k \ll\eta$. Let $T$ be a tree on $n$ vertices with $t \in V(T)$. Then, $T$ contains induced subgraphs $T_0\subset T_1\subset T_2\subset T_3=T$, such that $T_2$ is a tree, and the following hold.
  \begin{enumerate}[label = \textbf{\emph{P\arabic{enumi}}}]
    \item $|T_0|\leq \eta n$ and $t \in V(T_0)$.\label{cond1}
    \item $T_1$ is formed from $T_0$ by the vertex-disjoint addition of trees, $S_v$, $v\in V(T_0)$, so that, for each $v\in V(T_0)$, $S_v-v$ is a forest consisting of trees of size at most $K$.\label{cond2}
    \item $T_2$ is formed from $T_1$ by the addition of trees with size at least $k$ and at most $K$ attached to $T_1$ with exactly two bare paths of length 2.\label{cond3}
    \item $|T_3|-|T_2|\leq \eta n$.\label{cond4}
  \end{enumerate}
\end{lemma}

\begin{proof}
Take $\eps$ and $k'$ such that $1/K \ll \eps \ll 1/{k'} \ll 1/k$. We start by finding a subtree $T'$ of $T$ which includes $t$ and has few leaves, and is such that $T-V(T')$ is a forest of components with size at most $K$. We do this by including in $T'$ every vertex which appears on the path in $T$ from $t$ to many other vertices. That is, for each $v \in V(T)$, let $w(v)$ be the number of vertices $u\in V(T)$ whose path from $t$ to $u$ includes $v$ (in particular, $v$ is such a vertex). Let $T'$ be the subgraph of $T$ induced on all the vertices $v\in V(T)$ with $w(v) \geq K+1$. 
 
For each $v\in V(T')$, let $S_v$ be the tree containing $v$ in $T-(V(T')\setminus \{v\})$. Note that $S_v-v$ is a forest which consists of trees with at most $K$ vertices. Indeed, suppose $T''$ is a tree in $S_v - v$, and let $v'$ be the neighbour of $v$ in $T''$. Since every path from a vertex $u\in V(T'')$ to $t$ in $T$ goes through $v'$ (and then $v$), we have that $K \geq w(v') \geq |T''|$ (and, in fact, the final inequality is an equality). Thus, as $v' \notin V(T')$, $|T''|\leq K$. Observe further that, for any leaf $v$ of $T'$, $\size{S_v-v} = w(v)- 1 \geq K$, and, therefore, $T'$ can have at most $n/K \leq \eps n$ leaves.

Thus, by \cref{lem:few_leaves_many_bare_paths}, for some $m\leq n/(k'+1)$, $T'$ contains vertex disjoint bare paths $P_1,\dots,P_{m}$ with length $k'$ such that $t\notin V(P_i)$ for each $i\in [k']$ and 
\begin{equation}\label{eqn:last}
|T'-P_1-\dots-P_m|\leq 6k' \cdot \eps n + 2n/(k'+1) + k'+1 \leq \eta n/4.
\end{equation}
For each path $P_i$, $i\in [m]$, if possible, find within $P_i$ a path $P_i'$ with length at least $k'-2\eta^3k'$, such that, labelling its endvertices $x_i$ and $y_i$ the following hold.
\begin{enumerate}[label = (\roman{enumi})]
 \item For each $x_i$ and $y_i$, $\size{S_{x_i}-x_i} \leq \eta k'/4$ and $\size{S_{y_i}-y_i} \leq \eta k'/4$.
 \item Letting $Q_i$ be the component of $T-\{x_i,y_i\}$ containing $P_i'-\{x_i,y_i\}$, we have $|Q_i|\leq K$.
\end{enumerate}

Say, with relabelling, these paths are $P'_1,\ldots,P'_{m'}$. We will show that $m'\geq m-\eta n/2k'$.
Note first that the number of $i\in [m]$ with no vertices $x_i$ and $y_i\in V(P_i)$ respectively within $\eta^3k'$ of the two endvertices of $P_i$, such that each of $x_i$ and $y_i$ had a forest with at most $\eta k'/4$ vertices deleted from them, is at most $n/(\eta^3k'\cdot \eta k'/4)\leq \eta n/4k'$. Note further that the number of $i\in [m]$ with at least $K$ vertices in $V(Q_i)$ is at most $n/K\leq \eta n/4k'$. Therefore, we can find such a path $P'_i$ for all but at most $\eta n/2k'$ values of $i\in [m]$, so that $m'\geq m-\eta n/2k'$.

Let $T_0 = T[V(T') \setminus (\bigcup_{i \in [m']} V(P_i'))$]. We will show that $|T_0|\leq \eta n$. Note that, for each $i\in [m']$, $\size{V(P_i)\setminus V(P_i')}\leq 2\eta^3 k'$. Therefore, as $m\leq n/k'$,
\[
|T_0| \leq \size{T' - P'_1 - \dots - P'_{m'}}\leq |T'-P_1-\dots-P_m|+k'\cdot \eta n/2k'+ m\cdot 2\eta^3 k'\overset{\eqref{eqn:last}}\leq \eta n.
\]
Furthermore, clearly $t \in V(T_0)$, and thus \emph{\ref{cond1}} holds.

Let $T_1=T[V(T_0)\cup (\bigcup_{v\in V(T_0)}V(S_v))]$. Recall that for each $v\in V(T')$, $S_v-v$ is a forest which consists of trees with at most $K$ vertices. Therefore, \emph{\ref{cond2}} holds.

Let $T_2 = T[V(T_1)\cup (\bigcup_{i\in [m']}(\{x_i,y_i\}\cup V(Q_i)))]$, and let $T_3 = T$. Here, we obtain $T_2$ by attaching the trees $Q_i$ to vertices of $T_1$ by two bare paths of length 2, which have middle vertices given by the vertices $x_i$ and $y_i$. Since $Q_i$ contains the path $P'_i$ for every $i \in [m']$, each of these trees contain at least $k'-2\eta^3k' -2 \geq k$ vertices. On the other hand, by (ii), $|Q_i|\leq K$ for each $i\in [m']$ and so each such tree has size at most $K$. Therefore, \emph{\ref{cond3}} holds. Furthermore, the only missing vertices from $T$ are those in $S_v-v$ for each $v\in \{x_i,y_i:i\in [m']\}$, and thus $T_2$ is a tree. For each such $v$, $|S_v|\leq \eta k'/4$ by (i). Therefore, $|T_3|-|T_2|\leq (n/k')\cdot (2\eta k'/4)\leq \eta n$, and hence \emph{\ref{cond4}} holds.
\end{proof}

\subsection{Matchings between random sets}\label{sec:match}

With high probability, any random subset of vertices in the digraph in Theorem~\ref{full} satisfies a similar minimum semidegree condition, as follows.

\begin{lemma}\label{lem:degree_condition_inherited_by_random_subsets}
 Let $1/n\ll c,\alpha$, and suppose $D$ is an $n$-vertex digraph with $\delta^0(D)\geq (1/2+\alpha)n$. Let $A \subseteq V(D)$ be chosen uniformly at random subject to $\size{A}=cn$. Then, with high probability, for every vertex $v \in V(D)$, we have $\size{N_D^\pm(v,A)} \geq (1/2+\alpha/2)\size{A}$. 
\end{lemma}
\begin{proof}
Let $v$ be an arbitrary vertex of $D$ and let $A \subseteq V(D)$ be a uniformly random subset with $\size{A} = cn$. For $\diamond \in \{ +, -\}$, we let $Z^\diamond_v$ be the random variable which measures $\size{ N^\diamond(v) \cap A}$. Then $Z^\diamond_v$ has hypergeometric distribution with expectation
\[
\E[Z_v^\diamond] = \frac{\size{N^\diamond(v)} \size{A}}{n} \geq \left( \frac{1}{2} + \alpha \right)cn.
\]
Therefore, by \cref{lem:chernoff}, we have
\begin{align*}
\Prob \left[ \size{Z_v^\diamond - \E[Z_x^\diamond]} > \frac{\alpha/2}{1/2 + \alpha}(1/2 + \alpha)cn \right] &\leq 2 \exp \left(- \left(\frac{\alpha/2}{1/2 + \alpha}\right)^2\frac{(1/2 + \alpha)cn}{3} \right)\\
&= 2 \exp \left( \frac{-\alpha^2cn}{6+12\alpha} \right).  
\end{align*}
Then, applying a union bound, with probability at least $1-2n \exp \lb - \alpha^2 c n / (6+12\alpha) \rb =1-o(1)$, we have that $Z^\diamond_v \geq (1/2 + \alpha/2) \size{A}$ for each $\diamond\in \{+,-\}$ and $v\in V(D)$.
\end{proof}

The following digraph version of Hall's matching criterion implies a matching exists, as follows directly from the same result for undirected graphs.

\begin{lemma}\label{lem:halls}
Let $D$ be a bipartite digraph with vertex classes $A$ and $B$, and let $\diamond \in \{+, -\}$. Suppose that for every $S \subset A$, $\size{N_D^\diamond(S,B)} \geq \size{S}$. Then there is a $\diamond$-matching from $A$ into $B$ which covers $A$. 
\end{lemma}

We will refer to the condition in \cref{lem:halls} as \emph{Hall's criterion}. In combination with Lemma~\ref{lem:degree_condition_inherited_by_random_subsets}, \cref{lem:halls} shows that with high probability there is a perfect matching between a large random pair of disjoint equal-sized vertex subsets in the digraph, as follows.

\begin{proposition}\label{prop:matching_between_sets}
 Let $1/n\ll p,\alpha$, and suppose $D$ is an $n$-vertex digraph with $\delta^0(D)\geq (1/2+\alpha)n$. Let $A,B$ be chosen uniformly at random from all disjoint pairs of subsets of $V(D)$, each with size $pn$, and let $\diamond \in \{+, -\}$. Then, with high probability, there is a perfect $\diamond$-matching from $A$ into~$B$.
\end{proposition}

\begin{proof}
By \cref{lem:degree_condition_inherited_by_random_subsets}, with high probability we can assume the following. For all $v \in A$, we have $\size{N^\pm(v,B)} \geq (1/2 + \alpha/2) \size{B}$, and, for all $v\in B$, we have $\size{N^\pm(v,A)} \geq (1/2 + \alpha/2) \size{A}$. We will now show that Hall's criterion holds.

Let $S \subseteq A$, such that $S\neq \emptyset$ and $\size{S} \leq (1/2 + \alpha/2) pn$, and let $x \in S$. Then, $\size{N^\diamond(S,B)} \geq \size{N^\diamond(x,B)} \geq (1/2 + \alpha/2) pn \geq \size{S}$, so Hall's condition is trivially satisfied. Now take $S \subseteq A$, $\size{S} > (1/2 + \alpha/2)pn$, and assume for a contradiction that $\size{N^\diamond(S,B)} < \size{S}$. Then in particular, $B \setminus N^\diamond(S,B) \neq \emptyset$. Take $b \in B \setminus N^\diamond(S,B)$, and let $\circ \in \{+,-\}$ be such that $\circ \neq \diamond$. We have $\size{N^\circ(b,A)} \geq (1/2 + \alpha/2)pn$. However, since $b \not\in N^\diamond(S,B)$, we have $N^\circ(b,A)\cap S = \emptyset$. So,
\[
    pn = \size{A} \geq \size{N^\circ(b,A)} + \size{S}\geq (1/2 + \alpha/2)pn + (1/2+\alpha/2)pn = (1 + \alpha)pn > pn,
\]
giving a contradiction. Thus, Hall's criterion is satisfied for all $S \subseteq A$ and so, since $\size{A} = \size{B}$, by Lemma~\ref{lem:halls}, there is a perfect $\diamond$-matching from $A$ into $B$.
\end{proof}

We use \cref{prop:matching_between_sets} to embed many vertex disjoint small trees, via the following two lemmas. In \cref{lem:linearly_many_copies_of_a_tree}, we embed linearly many copies of a given constant-sized tree into specified subsets of our digraph. In \cref{lem:linearly_many_copies_of_a_tree}, we embed a forest of constant-sized trees covering almost all the vertices in our digraph.

\begin{lemma}\label{lem:linearly_many_copies_of_a_tree} Let $1/n\ll 1/K,p,\alpha$ with $pK\leq 1$. Suppose $T$ is an oriented $K$-vertex tree containing $t\in V(T)$. Let $D$ be an $n$-vertex digraph with $\delta^0(D)\geq (1/2+\alpha)n$. Let $V_1,V_2$ be vertex disjoint subsets of $V(D)$ chosen uniformly at random subject to $\size{V_1}=pn$ and $\size{V_2}=(K-1)pn$.

Then, with high probability, $D[V_1\cup V_2]$ contains $pn$ vertex disjoint copies of $T$, in which $t$ is copied into $V_1$ in each copy of $T$.
\end{lemma}

\begin{proof}  Let $V_1 = U_1$, and let $U_2 \cup \dots \cup U_K$ be a partition of $V_2$ chosen uniformly at random so that $\size{U_i}=pn$ for each $i \in \{2, \dots, K\}$. Note that the distribution of any pair of sets $U_i,U_j$ with $1\leq i<j\leq K$  is that of two disjoint vertex sets with size $pn$ in $V(D)$, uniformly at random drawn from all such pairs.

Label the vertices of $T$ by $t_1, \dots, t_K$ so that $t_1=t$ and $T[\{t_1, \dots, t_i\}]$ is a tree for each $i \in \{1, \dots, K\}$.  For each $i \in \{2, \dots, K\}$, let $j_i \in \{1, \dots, i-1\}$ be such that $t_{j_i}$ is the in- or out-neighbour in $T[\{t_1, \dots, t_{i-1}\}]$ of the vertex $t_i$, and let $\diamond_i \in \{+, -\}$ be such that $t_i \in N_T^{\diamond_i}(t_{j_i})$. 

Now by \cref{prop:matching_between_sets}, for each $i \in \{2, \dots, K\}$, with high probability, we can find a $\diamond_i$-matching from $U_{j_i}$ into $U_{i}$. By applying a union bound, we see that, with high probability, for every $i \in \{2, \dots, K\}$, there is a $\diamond_i$-matching, $M_i$ say, from $U_{j_i}$ into $U_i$.

Note that the union of these matchings, $\bigcup_{2\leq i\leq K}M_i\subset D[V_1\cup V_2]$ is the disjoint union of $pn$ copies of $T$, in which, for each $i\in [K]$, the copy of $t_i$ is in $V_i$. Thus, in each of these $pn$ copies of $T$, $t=t_1$ is copied into $V_1=U_1$, as required.
\end{proof}

\begin{lemma}\label{lem:forest_of_small_trees} Let $1/n\ll 1/K,\eps,\alpha$ and suppose $F$ is a digraph with at most $(1-\eps)n$ vertices which is the disjoint union of trees with size at most $K$. Let $D$ be an $n$-vertex digraph with $\delta^0(D)\geq (1/2+\alpha)n$. Then, with high probability, $D$ contains a copy of $F$.
\end{lemma}

\begin{proof}
Arrange the components of $F$ into isomorphic classes of trees $\R_1, \dots, \R_\ell$, noting that we may take $\ell \leq (2K)^{K-1}$. For each $i\in [\ell]$, let $t_i=|\R_i|$ and let $s_i$ be the size of each component in $\R_i$. Uniformly at random, take, in $V(D)$, disjoint subsets $V_{i,1}$ and $V_{i,2}$, $i\in [\ell]$, with  $\size{V_{i,1}}= p_in$ and $\size{V_{i,2}}= (s_i-1)p_i n$, where $p_i = t_i/n + \eps / \ell s_i$, for each $i\in [\ell]$. Note that this is possible, since
\[
    \sum_{i=1}^{\ell} s_i p_i n = \sum_{i=1}^{\ell} \lb  s_i t_i + \frac{\eps n}{\ell} \rb  \leq n.
\]
For each $i\in [\ell]$, we can apply \cref{lem:linearly_many_copies_of_a_tree} to show that, with high probability, there are $p_i n$ copies of the underlying tree of $\R_i$ in $D_i = D[V_{i,1} \cup V_{i,2}]$. Since $p_i n \geq t_i$, this implies that with high probability, we can find a copy of $\R_i$ in $D_i$ for each $i\in [\ell]$. By applying a union bound and using that $1/n\ll 1/\ell$, we have, with high probability, that there is a copy of $F$ in $D$.
\end{proof}

%% file: 03_almost_spanning_trees.tex
\section{Almost-spanning trees}\label{sec:almost}

The key aim of this section is to prove \cref{almost}, that is, to prove we can embed an almost-spanning tree $T$ in our digraph. By \cref{lem:decomposing_trees}, we can find $T_0 \subset T_1 \subset T_2 \subset T_3=T$, satisfying \emph{\ref{cond1}} to \emph{\ref{cond4}}. In \cref{sec:csstars}, we show that we can embed $T_1$. In \cref{sec:CSaspaths}, we show that we can embed $T_2 \setminus T_1$, and $T_3 \setminus T_2$. We conclude in \cref{sec:proofofalmost} by combining this to obtain an embedding of $T$. 

\subsection{Embedding constant-sized trees as stars}\label{sec:csstars}
As sketched in Section~\ref{subsec:sketch}, we will embed $T_0$ randomly, leaf by leaf, using a guide set to embed each new vertex. Each guide set has an accompanying guide graph, which we later use to find a matching. The property of the guide graph that we use to find the matching is that it is \emph{skew-bounded}, as follows.

\begin{defn}\label{def:skew_bounded}
A digraph $D$ with vertex sets $A,B\subset V(D)$ is \emph{$(a,b,\diamond)$-skew-bounded on $(A,B)$} if $d_D^\diamond(v,B)\geq a$ for each $v\in A$ and $d_D^\circ(v,A)\leq b$ for each $v\in B$, where $\circ\in \{+,-\}$ and $\circ\neq \diamond$.
\end{defn}

This property can imply a matching exists via Hall's criterion, as follows.

\begin{proposition}\label{skewhall}
 Let $a\geq b$ and $\diamond\in \{+,-\}$. Suppose $D$ is a digraph containing disjoint vertex sets $A,B\subset V(D)$, such that $D$ is  $(a,b,\diamond)$-skew-bounded on $(A,B)$. Then, there is a $\diamond$-matching from $A$ into $B$ in $D$ which covers $A$.
\end{proposition}
\begin{proof}
Let $U\subset A$. As $D$ is $(a,b,\diamond)$-skew-bounded on $(A,B)$, there are at least $a\size{U}$ and at most $b\size{N_D^\diamond(U,B)}$ $\diamond$-edges from $U$ to $N_D^\diamond(U,B)$. Thus, $\size{N_D^\diamond(U,B)}\geq a\size{U}/b\geq \size{U}$. Therefore, by Lemma~\ref{lem:halls}, there is a $\diamond$-matching from $A$ into $B$ which covers $A$.
\end{proof}


In the following lemmas, we find our guide sets and guide graphs. We start by finding in $D$, for each $v\in V(D)$ and $\diamond\in \{+,-\}$, a guide set $A$ and guide graphs which are skew-bounded on $(A,V(D))$.


\begin{lemma}\label{lem-skew-1}
Let $1/n\ll \eps\ll \alpha,\eta\leq 1$ and $1/n\ll\mu\leq \alpha^2/2$. Let $D$ be an $n$-vertex digraph with $\delta^0(D)\geq (1/2+\alpha)n$, let $v\in V(D)$ and let $\diamond\in \{+,-\}$.

Then, there is a set $A\subset N_D^\diamond(v)$ with size $\mu n$ and digraphs $H^+,H^-\subset D$ such that, for each $\circ\in \{+,-\}$, $H^\circ$ is $(\eps n,(1+\eta)\mu \eps n,\circ)$-skew-bounded on $(A,V(D))$.
\end{lemma}
\begin{proof} We start by showing that we can label the vertices of $V(D)$ as $V(D)=\{x_1,\ldots,x_n\}=\{y_1,\ldots,y_n\}$ so that, for each $i\in [n]$,
\begin{equation}\label{xymatch}
|N^-_D(x_i)\cap N^\diamond_D(v)\cap N^+_D(y_i)|\geq \alpha^2 n.
\end{equation}

To do this, create an auxiliary graph, as follows. For each $w\in V(D)$, create distinct new vertices $w^-$ and $w^+$, and let $V^+=\{w^+:w\in V(D)\}$ and $V^-=\{w^-:w\in V(D)\}$. Consider the auxiliary bipartite graph $H$ with vertex set $V^+\cup V^-$, where for each $x,y\in V(D)$, there is an edge between $x^+$ and $y^-$ if and only if $\size{N^-_D(x)\cap N_D^\diamond(v)\cap N^+_D(y)}\geq \alpha^2 n$.

\begin{claim}\label{claimforhall}
$\delta(H)\geq (1/2+\alpha/2)n$.
\end{claim}
\begin{proof}[Proof of \cref{claimforhall}]
Let $x\in V(D)$. We have  $|N^-_D(x)\cap N^\diamond_D(v)|\geq n-(n-d_D^-(x))-(n-d_D^\diamond(v))\geq 2\alpha n$. Let $B=N^-_D(x)\cap N^\diamond_D(v)$ and $Y=\{y\in V(D):|N^+_D(y)\cap B|\geq \alpha^2 n\}$, and note that $d_H(x^+)= |Y|$.

 For each $u\in B$, we have $|N^-_D(u)|\geq (1/2+\alpha)n$, and thus $e_D(V(D),B)\geq (1/2+\alpha)|B|n$. By the choice of $Y$, we have $e_D(V(D),B)\leq |Y||B|+\alpha^2 n^2$. Therefore, as, in addition, $2\alpha n\leq |B|$, we have
\[
(1/2+\alpha)|B|n\leq |Y||B|+\alpha^2 n^2\leq |Y||B|+\alpha |B|n/2.
\]
Thus, $ (1/2+\alpha/2)|B|n\leq |Y||B|$, so that $|Y|\geq (1/2+\alpha /2)n$. Therefore, $d_H(x^+)=|Y|\geq (1/2+\alpha/2)n$.

A similar argument, with the signs reversed, shows that $d_H(y^-)\geq (1/2+\alpha/2)n$ for each $y\in V(D)$, completing the proof of the claim.
\end{proof}
As in the proof of Proposition~\ref{prop:matching_between_sets}, Claim~\ref{claimforhall} easily implies that Hall's criterion is satisfied, so that there is a matching from $V^+$ to $V^-$ in $H$. That is, we can label the vertices of $V(D)$ as $V(D)=\{x_1,\ldots,x_n\}=\{y_1,\ldots,y_n\}$ so that, for each $i\in [n]$, \eqref{xymatch} holds.

We will now show by induction that, for each $0\leq i\leq \mu n$, there is a set $A_i\subset N_D^\diamond(v)$ with size $i$ and graphs $H_i^+,H_i^-\subset D$ such that, for each $\circ\in \{+,-\}$, $H_i^\circ$ is $(\eps n,(1+\eta)\mu \eps n,\circ)$-skew-bounded on $(A_i,V(D))$, $e(H_i^\circ)=i\eps n$, and, for each $j\in [n]$, $d_{H^+_i}^-(x_j)=d_{H^-_i}^+(y_j)$.

Note that if $A_0=\emptyset$ and if $H^+_0$, $H^-_0$ have no edges and vertex set $V(D)$, then the conditions hold, so assume that $0\leq i<\mu n$ and we have $A_i\subset N_D^\diamond(v)$ and $H_i^+, H_i^-\subset D$ as described.

Let $J_i\subset [n]$ be the set of $j\in [n]$ for which $d_{H^+_i}^-(x_j)=d_{H^-_i}^+(y_j)\leq (1+\eta/2)\mu \eps n$. Note that, as $e(H^+_i)=e(H^-_i)= i\eps n\leq \mu \eps n^2$, we have
\[
(n-|J_i|)(1+\eta/2)\mu\eps n\leq \mu \eps n^2.
\]
Thus, as $\eta\leq 1$, $(n-|J_i|)\leq n/(1+\eta/2)\leq  n(1-\eta/4)$, so that $|J_i|\geq \eta n/4$.

For each $j\in J_i$, let $W_{i,j}=(N^-_D(x_j)\cap N^\diamond_D(v)\cap N^+_D(y_j))\setminus A_i$, noting that, by \eqref{xymatch}, $|W_{i,j}|\geq \alpha^2n-i> \alpha^2n-\mu n\geq \alpha^2n/2$. By averaging, choose some $w_i\in V(D)$ such that
\[
|\{j\in J_i:w_i\in W_{i,j}\}|\geq \frac{\sum_{j\in J_i}|W_{i,j}|}{n}\geq \frac{\eta n/4\cdot \alpha^2n/2}{n}\geq \eps n,
\]
using that $\alpha,\eta\gg \eps$. Choose a set $J_i'\subset \{j\in J_i:w_i\in W_{i,j}\}$ with size $\eps n$. Let $A_{i+1}=A_i\cup \{w_i\}$. Let $H^+_{i+1}$ be the digraph $H^+_i$ with edges $w_ix_j$, $j\in J'_i$, added. Note that, as $d^-_{H^+_i}(x_j)\leq (1+\eta/2)\mu \eps n$ for each $j\in J'_i$, $H^+_{i+1}$ is $(\eps n,(1+\eta)\mu \eps n ,+)$-skew-bounded on $(A_{i+1},V(D))$. Furthermore, by the definition of $W_{i,j}$, the edges added to $H^+_i$ are in $D$, and therefore $H^+_{i+1}\subset D$.

 Let $H^-_{i+1}$ be the digraph $H^-_i$ with the edges $y_jw_i$, $j\in J'_i$, added. Note that, similarly, $H^-_{i+1}$ is $(\eps n,(1+\eta)\mu \eps n,-)$-skew-bounded on $(A_{i+1},V(D))$. Finally, noting that $A_{i+1}$ has size $i+1$, that $e(H_{i+1}^+)=e(H_{i+1}^-)=(i+1)\eps n$ and that, for each $j\in [n]$, $d^-_{H_{i+1}^+}(x_j)=d^+_{H_{i+1}^-}(y_j)$, completes the inductive step, and hence the proof.
\end{proof}


We now show that the guide sets and guide graphs found by \cref{lem-skew-1} have a similar skew-bounded property when restricted to random vertex subsets, as follows.


\begin{lemma}\label{lem-skew-2} Let $1/n\ll \eps\ll \alpha,\eta\leq 1$ and $1/n\ll 1/k,p_0,p_1,\ldots,p_k\leq 1$. Let $\mu= \alpha^2p_0/4$. Let $D$ be an $n$-vertex digraph with $\delta^0(D)\geq (1/2+\alpha)n$. Let $V_0,V_1,\ldots,V_k\subset V(D)$ be disjoint random sets chosen uniformly at random subject to $|V_i|=p_in$ for each $i\in \{0,\ldots,k\}$.

Then, with high probability, for each $v\in V(D)$ and $\diamond\in \{+,-\}$, there is a set $A_{v,\diamond}\subset N_D^\diamond(v)\cap V_0$ with size $\mu n$ and digraphs $H_{v,\diamond}^{\circ}\subset D$, $\circ\in \{+,-\}$, such that, for each $\circ\in \{+,-\}$ and $i\in [k]$, $H_{v,\diamond}^{\circ}$ is $(\eps p_in,(1+\eta)\eps \mu n,\circ)$-skew-bounded on $(A_{v,\diamond},V_i)$.
\end{lemma}
\begin{proof} By Lemma~\ref{lem-skew-1}, applied with $\eps'=(1+\eta/4)\eps$, $\eta'=\eta/4$ and $\mu'=(1+\eta/4)\alpha^2/4$, for each $v\in V(D)$ and $\diamond\in \{+,-\}$, there is a set $\bar{A}_{v,\diamond}\subset N_D^\diamond(v)$ with size $(1+\eta/4)\alpha^2 n/4$ and digraphs $H_{v,\diamond}^+,H_{v,\diamond}^-\subset D$ such that, for each $\circ\in \{+,-\}$, $H_{v,\diamond}^\circ$ is $((1+\eta/4)\eps n,(1+\eta/4)^3\eps\alpha^2 n/4,\circ)$-skew-bounded on $(\bar{A}_{v,\diamond},V(D))$.

Select $V_0,V_1,\ldots,V_k\subset V(D)$ according to the distribution in the lemma.
Using Lemma~\ref{lem:chernoff}, and a union bound, we have that, with high probability, the following hold.
\begin{enumerate}[label = \textbf{Q\arabic{enumi}}]
\item For each $v\in V(D)$ and $\diamond\in \{+,-\}$, $|\bar{A}_{v,\diamond}\cap V_0|\geq \alpha^2p_0 n/4=\mu n$.\label{ham1}
\item For each $v\in V(D)$, $\diamond,\circ\in \{+,-\}$, and $w\in \bar{A}_{v,\diamond}$, $|N^\circ_{H^\circ_{v,\diamond}}(w,V_i)|\geq \eps p_in$.\label{ham2}
\item For each $v\in V(D)$, $\diamond,\circ\in \{+,-\}$, and $w\in V(D)$, $|N^{\bar{\circ}}_{H^\circ_{v,\diamond}}(w,\bar{A}_{v,\diamond})\cap V_0)|\leq (1+\eta)\eps\alpha^2 p_0n/4=(1+\eta)\eps \mu n$, where $\bar{\circ}\in \{+,-\}$ is such that $\bar{\circ}\neq \circ$.\label{ham3}
\end{enumerate}

Indeed, by Lemma~\ref{lem:chernoff}, as $\eps,\eta,\alpha,p_0,p_1,\ldots,p_k\gg 1/n$, for any instance of $v\in V(D)$, $\diamond,\circ\in \{+,-\}$, and $w\in V(D)$, the property \ref{ham1} above holds with probability $1-\exp(-\Omega(n))$, and the same is true for \ref{ham2} and \ref{ham3}. Therefore, by a union bound, with high probability, the properties \ref{ham1}, \ref{ham2} and \ref{ham3} hold.

Now, for each $v\in V(D)$ and $\diamond\in \{+,-\}$, using \ref{ham1}, choose $A_{v,\diamond}\subset \bar{A}_{v,\diamond}\cap V_0$ with $|A_{v,\diamond}|=\mu n$.
By \ref{ham2} and \ref{ham3}, we have, for each $\circ\in \{+,-\}$ and $i\in [k]$, that $H_{v,\diamond}^\circ$ is $(\eps p_in,(1+\eta)\eps\mu n,\circ)$-skew-bounded on $(A_{v,\diamond},V_i)$, as required,
\end{proof}


We will now use the guide sets produced by Lemma~\ref{lem-skew-2} to randomly embed $T_0$, the small core of the original tree, and then use the guide graphs to find matchings from certain subsets of the image of the embedding to other random sets, as follows.

\begin{lemma}\label{lem-skew-3}
Let $1/n \ll c\ll \beta\ll \eta,q,\alpha\leq 1$ and $1/n\ll c\ll p\ll 1/m$. Let $D$ be an $n$-vertex digraph with $\delta^{0}(D)\geq (1/2+\alpha)n$.

Let $T$ be an oriented tree with $\Delta^\pm(T)\leq cn/\log n$ consisting of a subtree $T_0\subset T$ with $|T_0|\leq \beta n$, such that every vertex in $V(T)\setminus V(T_0)$ is attached as a leaf to $T_0$. Let $t\in V(T_0)$. Let $U_0=V(T_0)$ and let $U_1\cup \ldots\cup U_m$ be a partition of $V(T)\setminus V(T_0)$ such that, for each $i\in [m]$, either $e_{T}(V(T_0),U_i)=0$ or $e_{T}(U_i,V(T_0))=0$. Let $V_0,V_1,\ldots,V_m\subset V(D)$ be disjoint random sets chosen uniformly at random subject to $|V_0|=qn$, and, for each $i\in [m]$, $|V_i|= \lfloor(1+\eta)|U_i|\rfloor+p n$.\footnote{Note that this gives an implicit bound on $\size{T}$.}

Then, with high probability, for each $s\in V_0$, there is an embedding of $T$ into $D$ such that $t$ is embedded to $s$, and, for each $i\in \{0,1,\ldots,m\}$, $U_i$ is embedded into $V_i$.
\end{lemma}
\begin{proof} Choose $\eps$ such that $\beta\ll\eps\ll \eta,q,\alpha$. For each $j\in [m]$, let $p_j=(\lfloor(1+\eta)|U_j|\rfloor/n)+p$.
Choose $V_0,V_1,\ldots,V_m$ according to the distribution in the lemma. By Lemma~\ref{lem-skew-2} applied with $\eta' = \eta/2$ and $p_0=q$, with high probability, for each $v\in V(D)$ and $\diamond\in \{+,-\}$, there is
\begin{enumerate}[label = \textbf{R\arabic{enumi}}]
\item a set $A_{v,\diamond}\subset N_D^\diamond(v)\cap V_0$ with size $q\alpha^2n/4$, and\label{theabove}
\item  digraphs $H_{v,\diamond}^{\circ}\subset D$, $\circ\in\{+,-\}$, such that, for each $j\in [m]$ and $\circ\in \{+,-\}$, $H_{v,\diamond}^{\circ}$ is $(\eps p_jn,(1+\eta/2)\eps q\alpha^2n/4,\circ)$-skew-bounded on $(A_{v,\diamond},V_j)$.\label{theabove2}
\end{enumerate}

We will now show that, given only \ref{theabove} and \ref{theabove2}, we can embed $T$ as required in the lemma for each $s\in V_0$. Let then $s\in V_0$. We will randomly embed $T_0$ into $D[V_0]$, as follows, before showing that, with high probability, it can be extended into the required copy of $T$. Let $\ell=|T_0|$ and label $V(T_0)=\{t_1,\ldots,t_\ell\}$, so that $t_1=t$ and $T_0[\{t_1,\ldots,t_i\}]$ is a tree for each $i\in [\ell]$. Let $s_1=s$ and embed $t_1$ to $s_1$. For each $i\in \{2,\ldots,\ell\}$ in turn, let $j_i\in \{1,\ldots,i-1\}$ be such that $t_{j_i}$ is the in- or out-neighbour of $t_i$ in $T_0[\{t_1,\ldots,t_i\}]$ and let $\diamond_i\in \{+,-\}$ be such $t_i\in N_{T_0}^{\diamond_i}(t_{j_i})$, and, uniformly at random, embed $t_i$ to $s_i\in A_{s_{j_i},\diamond_i}\setminus \{s_1,\ldots,s_{i-1}\}$. Such an embedding is possible since, for every $v \in V(D)$ and $\diamond \in \{+, -\}$, $\size{A_{v,\diamond}}$ is much larger than $\size{T_0}$ as $\beta\ll q,\alpha$.

\begin{claim} For each $j\in [m]$, with high probability, the embedding of $T_0$ can be extended to an embedding of $T[V(T_0)\cup U_j]$ by embedding $U_j$ into $V_j$.\label{claimembed}
\end{claim}
As $p\gg 1/n$, and $m\leq 1/p$, we can take a union bound over all $j\in [m]$, to show that, with high probability, for each $j\in [m]$, the embedding of $T_0$ can be extended to $T[V(T_0)\cup U_j]$ by embedding $U_j$ into $V_j$, and hence $T$ can be embedded as required in the lemma. Therefore, there is some choice of the embedding of $T_0$ for which this can be done. It is left then to prove Claim~\ref{claimembed}.

\medskip
\noindent\emph{Proof of Claim~\ref{claimembed}.} Let $j\in [m]$ and let $\circ_j\in \{+,-\}$ be such that all the edges from $V(T_0)$ to $U_j$ in $T$ are $\circ_j$-edges. For each $i\in [\ell]$, let $d_{j,i}=|N_T^{\circ_j}(t_i,U_j)|$. For each $i\in [\ell]$, take $d_{j,i}$ new vertices and call them $w_{j,i,i'}$, $i'\in [d_{j,i}]$. Let $W_j=\{w_{j,i,i'}:i\in [\ell],i'\in [d_{j,i}]\}$. Let $K_j$
be the directed graph with vertex set $W_j\cup V_j$, containing only $\circ_j$-edges from $W_j$ to $V_j$, and where, for each $i\in [\ell]$, $i'\in [d_{j,i}]$ and $v\in V_j$, there is a $\circ_j$-edge from $w_{j,i,i'}$ to $v$ in $K_j$ if, and only if, $s_iv\in E(H^{\circ_j}_{s_{j_i},\diamond_{i}})$.

We will show that, with high probability, $K_j$ is $(\eps p_jn,\eps p_j n,\circ_j)$-skew-bounded on $(W_j,V_j)$. This is enough to prove the claim, as, by Proposition~\ref{skewhall}, there is a $\circ_j$-matching from $W_j$ into $V_j$ in $K_j$ which covers $W_j$. Thus, we can label distinct vertices $v'_{j,i,i'}$, $i\in [\ell]$, $i'\in [d_{j,i}]$ in $V_j$ so that $w_{j,i,i'}v'_{j,i,i'}$, $i\in [\ell]$ and $i'\in [d_{j,i}]$, is a matching in $K_j$. For each $i\in [\ell]$, use the vertices $v'_{j,i,i'}$, $i'\in [d_{j,i}]$, to embed $d_{j,i}$ $\circ_j$-neighbours of $t_i$ in $U_j$ into $V_j$. This is possible as, by the definition of $K_j$ and $H^{\circ_j}_{s_{j_i},\diamond_{i}}$,
$s_iv'_{j,i,i'}$ is a $\circ_j$-edge in $D$. Therefore, this extends the embedding of $T_0$ to an embedding of $T_0\cup T[U_j]$ with $U_j$ embedded into $V_j$, as required.

Thus, it is sufficient to prove that, with high probability, $K_j$ is $(\eps p_jn,\eps p_j n,\circ_j)$-skew-bounded on $(W_j,V_j)$. Now, for each $i\in [\ell]$, $s_i\in A_{j_i,\diamond_i}$, and therefore $s_i$ has at least $\eps p_j n$ $\circ_j$-neighbours in $V_j$ in $H^{\circ_j}_{s_{j_i},\diamond_i}$ by \ref{theabove2}. Therefore, for each $i\in [\ell]$ and $i'\in [d_{j,i}]$, $w_{j,i,i'}$ has at least $\eps p_jn$ $\circ_j$-neighbours in $K_j$. That is, each $v\in W_j$ has at least $\eps p_jn$ $\circ_j$-neighbours in $K_j$. Thus, letting $\bar{\circ}_j\in \{+,-\}$ with $\bar{\circ}_j\neq \circ_j$, it is sufficient to prove that, for each $v\in V_j$, with probability $1-o(n^{-1})$, $d^{\bar{\circ}_j}_{K_j}(v,W_j)\leq \eps p_jn$.

Let then $v\in V_j$. For each $i\in [\ell]$, let
\[
X^{j,v}_i=\left\{\begin{array}{ll}
d_{j,i} & \quad\text{ if }s_iv\in E(H^{\circ_j}_{s_{j_i},\diamond_{i}})\\
0 & \quad\text{ otherwise,}
\end{array}\right.
\]
so that $d^{\bar{\circ}_j}_{K_j}(v,W_j)=\sum_{i\in [\ell]}X^{j,v}_i$.
Note that, when  $s_i\in A_{s_{j_i},\diamond_i}\setminus \{s_1,\ldots,s_{i-1}\}$ is chosen uniformly at random, by \ref{theabove} and \ref{theabove2}, and as $\beta\ll \eta,\alpha,q$ and $i\leq \ell\leq \beta n$, if $d_{i,j}>0$, then $X^{j,v}_i=d_{j,i}$ with probability at most
\[
\frac{d_{H^{\circ_j}_{s_{j_i},\diamond_i}}(v)}{|A_{s_{j_i},\diamond_i}\setminus \{s_1,\ldots,s_{i-1}\}|}\leq \frac{(1+\eta/2)\eps q\alpha^2n/4}{ q\alpha^2 n/4-(i-1)}\leq (1+\eta)\eps.
\]
Let $\gamma=(1+\eta)\eps$. Then, for each $i\in [\ell]$, $\E[X^{j,v}_i|X^{j,v}_1,\ldots,X^{j,v}_{i-1}]\leq \gamma\cdot d_{j,i}$.
Note that the inequality 
\begin{equation}\label{eq:quadratic_inequality}
a^2+b^2 \leq (a-1)^2+(b+1)^2    
\end{equation}
holds whenever $1 \leq a \leq b$. Repeated application of this inequality shows that $\sum_{i\in [\ell]}d_{j,i}^2$ is maximised when as many of the $d_{j,i}$ are maximised as possible. Therefore, as $d_{j,i}\leq cn/\log n$ for each $i\in [\ell]$, and $\sum_{i\in [\ell]}d_{j,i}\leq |U_j|\leq  n$, we have $\sum_{i\in [\ell]}d_{j,i}^2 \leq (n/(cn/\log n))(cn/\log n)^2 = cn^2/\log n$.
Using this and the fact that $|X^{j,v}_i - \gamma \cdot d_{j,i}| \leq d_{j,i}$ for each $i\in [\ell]$ and $c \ll p$, we can apply \cref{cor:azuma}~\ref{cor:azuma_sup} with $a_i = \gamma \cdot d_{j,i}$, $c_i = d_{j,i}$ and $t=p n/3$ to get
\[
\Prob\lsb\sum_{i\in [\ell]}X^{j,v}_i\geq \gamma\cdot\left(\sum_{i\in [\ell]}d_{j,i}\right) + p n/3\rsb \leq 2 \exp \lb \frac{-(pn/3)^2}{\sum_{i\in [\ell]}d_{j,i}^2} \rb \leq 2\exp \lb \frac{-p^2 \log n}{9c} \rb \leq o(n^{-1}).
\]
Thus, with probability $1-o(n^{-1})$, we have 
\begin{align*}
d^{\bar{\circ}_j}_{K_j}(v,W_j)&=\sum_{i\in [\ell]}X^{j,v}_i < \gamma\cdot\left(\sum_{i\in [\ell]}d_{j,i}\right)+p n/3 < \gamma|U_j|+p n/3\leq \gamma p_jn/(1+\eta)\\
&=(1+\eta)\eps p_jn/(1+\eta)=\eps p_j n,
\end{align*}
completing the proof of the claim, and hence the lemma.
\hspace{4.5cm}\hfill\qed
\end{proof}

Finally, by combining Lemma~\ref{lem-skew-3} and Lemma~\ref{lem:linearly_many_copies_of_a_tree}, we can prove Lemma~\ref{lem:CSstars}.

\begin{proof}[Proof of Lemma~\ref{lem:CSstars}] Let $p$ satisfy $1/n\ll c\ll p \ll \eps, 1/K$.
For each $j\in [\ell]$, let $s_j$ be the vertex of $S_j$ with an in- or out-neighbour in $V(T')$ in $T$. Let $\mathcal{R}$ be a maximal set of pairs $(R,r)$ for which $R$ is a directed tree with at most $K$ edges and $r\in V(R)$, such that the pairs $(R,r)$ are unique up to isomorphism. Let $m=|\mathcal{R}|$ and enumerate $\mathcal{R}$ as $(R_1,r_1),\ldots,(R_m,r_m)$. Note that $p\ll 1/m$.

Let $T''=T[V(T')\cup N^+_T(V(T'))\cup N^-_T(V(T'))]$. For each $i\in [m]$ and $\diamond \in \{+,-\}$, let $U_{i,\diamond}\subset V(T'')$ be the set of vertices $s_j$, $j\in [\ell]$, for which $(S_j,s_j)$ is isomorphic to  $(R_i,r_i)$ and the edge from $V(T')$ to $s_j$ in $T$ is a $\diamond$-edge.

In $V(D)$, take disjoint random sets $V_0$ and $V_{i,\diamond,j}$, $i\in [m]$, $\diamond\in \{+,-\}$ and $j\in \{1,2\}$, uniformly at random subject to the following.
\begin{itemize}
\item $|V_0|=\eps n/2$.
\item  For each $i\in [m]$ and $\diamond\in \{+,-\}$, we have that $|V_{i,\diamond,1}|=\lfloor(1+\eps/6)|U_{i,\diamond}|\rfloor+p n$ and $|V_{i,\diamond,2}|=(\lfloor(1+\eps/6)|U_{i,\diamond}|\rfloor+p n)(|R_i|-1)$.
\end{itemize}
Note that this is possible, as
\begin{align*}
|V_0|+\sum_{i\in [m],\diamond\in \{+,-\}}(|V_{i,\diamond,1}|+|V_{i,\diamond,2}|)&=|V_0|+\sum_{i\in [m],\diamond\in \{+,-\}}(\lfloor(1+\eps/6)|U_{i,\diamond}|\rfloor+p n)|R_i|
\\
&\leq \eps n/2+(1+\eps/6)\sum_{j\in [\ell]}|S_j|+\sum_{i\in [m]}2p n\cdot |R_i|\\
&\leq \eps n/2+(1+\eps/6)|T|+(2p n)\cdot m\cdot K
\leq n.
\end{align*}

Now, with probability $\eps /2$, $v\in V_0$. By Lemma~\ref{lem-skew-3}, with high probability, if $v\in V_0$, then there is an embedding of $T''$ into $D$ such that $t$ is embedded to $v$, $V(T')\subset V_0$, and, for each $i\in [m]$ and $\diamond\in \{+,-\}$, $U_{i,\diamond}$ is embedded into $V_{i,\diamond,1}$. By \cref{lem:linearly_many_copies_of_a_tree}, for each $i\in [m]$ and $\diamond\in \{+,-\}$, $D[V_{i,\diamond,1}\cup V_{i,\diamond,2}]$ contains $|V_{i,\diamond,1}|$ vertex disjoint copies of $R_i$, in which $r_i$ is copied into $V_{i,\diamond,1}$. For each $i\in [m]$ and $\diamond\in \{+,-\}$, add each copy of $R_i$ containing an embedded vertex of $U_{i,\diamond}$ to the embedding of $T''$. Note that this results in a copy of $T$.
\end{proof}

\subsection{Embedding constant-sized trees as paths}\label{sec:CSaspaths}

Given our decomposition $T_0\subset T_1\subset T_2\subset T_3=T$, we have now embedded $T_1$. We now embed the vertices from $V(T_2) \setminus V(T_1)$, recalling that we obtain $T_2$ from $T_1$ by adding constant-sized trees, where each tree is attached to $T_1$ by exactly two bare paths of length $2$. In the following lemma, we embed $T_2\setminus T_1$ so that the vertices in $V(T_2)\cap V(T_1)$ are embedded to preselected vertices (labelled $a_i,b_i$, $i\in [\ell]$). This allows us to extend our embedding of $T_1$ to one of $T_2$.

\begin{lemma}\label{lem:CSaspaths}
Let $1/n\ll 1/K\leq 1/k\ll \alpha,\eps$. Suppose $T$ is a forest formed of vertex-disjoint oriented trees $T_i$, $i\in [\ell]$, with at most $(1-\eps)n$ vertices in total, and so that $k\leq |T_i|\leq K$, for each $i\in [\ell]$, and each tree $T_i$ contains distinct vertices $r_i$ and $s_i$ which are leaves in $T_i$ whose neighbour has total in- and out-degree 2.

Suppose $D$ is an $n$-vertex digraph with $\delta^0(D)\geq (1/2+\alpha)n$, containing the distinct vertices $a_i$, $b_i$, $i\in [\ell]$. Then, $D$ contains a copy of $T$ in which, for each $i\in [\ell]$, $r_i$ is embedded to $a_i$ and $s_i$ is embedded to $b_i$.
\end{lemma}

\begin{proof}
Let $\beta$ be such that $1/k \ll \beta \ll \alpha,\eps$. For each $i \in [\ell]$, let $r_i'$ and $s_i'$ be the neighbours in $T_i$ of $r_i$ and $s_i$, respectively, and let $T_i' = T_i - \{ r_i, r_i', s_i, s_i'\}$. Let $T'$ be the forest composed of connected components $T_i'$, $i\in [\ell]$, so that $|T_i'|\leq (1-\eps)n$.
%
Let $A=\{a_i,b_i:i\in [\ell]\}$. Then $\size{A} = 2\ell \leq 2n/k$. Let $B \subset V(D) \setminus A$ be a random subset of vertices with $\size{B} = \beta n$.

Let $D' = D-A-B$. As $1/k,\beta\ll \alpha,\eps$, we have $\size{D'} \geq (1-\eps/4)n$ and $\delta^0(D') \geq (1/2 + \alpha/2)\size{D'}$. Since
\[
\size{T'} \leq (1-\eps)n\leq \frac{(1-\eps)}{(1-\eps/4)}|D'|\leq (1-\eps/2)\size{D'},
\]
we have, by 
\cref{lem:forest_of_small_trees}, with high probability we can find a copy, $S'$ say, of $T'$ inside $D'$.

Let $r_i''$ and $s_i''$ be the neighbours in $T'$ of $r_i'$ and $s_i'$, respectively, for each $i\in [\ell]$, and let $a_i''$ and $b_i''$ be the copy of $r_i''$ and $s_i''$ in $S'$, respectively.

\begin{claim}\label{claim:intersecting_neighbourhoods}
The following holds with high probability. For any pair of vertices $u, v \in V(D)$ and $\diamond, \circ \in \{ +, -\}$, we have that $\size{N^\diamond(u) \cap N^\circ(v) \cap B} \geq \alpha \beta n$.
\end{claim}
\begin{proof}[Proof of \cref{claim:intersecting_neighbourhoods}]
Let $u,v \in V(D)$ and $\diamond, \circ \in \{ +, -\}$. Note that,  by the semi-degree condition on $D$, $\size{N^\diamond(u) \cap N^\circ(v)} \geq 2\alpha n$, and hence $\size{N^\diamond(u) \cap N^\circ(v) \cap B}$ has a hypergeometric distribution with $\E\size{N^\diamond(u) \cap N^\circ(v) \cap B} \geq 2\alpha \beta n$. By \cref{lem:chernoff}, and
%
%
a union bound over all pairs $u,v \in D$ and $\diamond, \circ \in \{ +, -\}$, the statement in the claim thus holds with probability $1-o(1)$.
\end{proof}

Thus, with high probability, we can assume the property in the claim holds. Now, for each $i \in [\ell]$, embed $r_i$ and $s_i$ to $a_i$ and $b_i$, respectively. Let $\diamond_i, \circ_i, \diamond_i', \circ_i' \in \{+, -\}$ be such that $r_i' \in N^{\diamond_i}(r_i) \cap N^{\circ_i}(r_i'')$, and $s_i' \in N^{\diamond_i'}(s_i) \cap N^{\circ_i'}(s_i'')$. Greedily and disjointly, for each $i\in [r]$, embed $r_i'$ to a vertex in $N^{\diamond_i}(a_i) \cap N^{\circ_i}(a_i'') \cap B$ and embed $s_i'$ to a vertex in $N^{\diamond_i'}(b_i) \cap N^{\circ_i'}(b_i'') \cap B$. Note that this is possible, since, from the property in the claim we have, for each $i\in [r]$ 
\[
\size{N^{\diamond_i}(a_i) \cap N^{\circ_i}(a_i'') \cap B} ,\size{N^{\diamond_i'}(b_i) \cap N^{\circ_i'}(b_i'') \cap B}\geq \alpha \beta n \geq \frac{2n}{k}\geq 2r.
\]
This completes the embedding of $T$ with the property required in the lemma.
\end{proof}

\subsection{Proof of Theorem~\ref{almost}}\label{sec:proofofalmost}
We now combine \cref{lem:CSstars} and \cref{lem:CSaspaths} to find a copy of any almost-spanning tree. 

\begin{proof}[Proof of Theorem~\ref{almost}]
Take $K,k$ and $\eta$ so that $c \ll 1/K \ll 1/k \ll \eta \ll \eps, \alpha$. Let $D$ be an $n$-vertex graph with $\delta^0(D) \geq (1/2 + \alpha) n$. Let $T$ be an oriented tree on at most $(1-\eps)n$ vertices with $\Delta^\pm(T) \leq cn/ \log n$. By \cref{lem:decomposing_trees}, we can find forests $T_0 \subset T_1 \subset T_2 \subset T_3= T$ satisfying \emph{\ref{cond1}} to \emph{\ref{cond4}}. Randomly partition $V(D)$ into three parts, $V(D) = V_1 \cup V_2 \cup V_3$ so that $\size{V_1} = \size{T_1} + \eps n /3$, $\size{V_2} = \size{T_2} - \size{T_1} + \eps n /3$, and $\size{V_3} = \size{T} - \size{T_2}+ \eps n /3$. Note that, with probability at least $\eps/3$, we have $v \in V_1$.

By applying \cref{lem:degree_condition_inherited_by_random_subsets} with $A = V_1$, with high probability we have $\delta^0(D[V_1]) \geq (1/2 + \alpha/2) \size{V_1}$. Thus, by applying \cref{lem:CSstars} to $D = D[V_1]$ and $T=T_1$, we can find a copy of $T_1$ in $V_1$ in which $t$ is copied to $v$. By \emph{\ref{cond3}}, for some $\ell\in \N$,  $T_2$ is formed from $T_1$ by the addition of trees $F_i$, $i \in [\ell]$, where $k \leq \size{F_i} \leq K$, which are each attached to $T_1$ by exactly two bare paths of length 2, $P_i$ and $Q_i$ say. For each $i \in [\ell]$, let $p_i$ and $q_i$ be the endpoint of $P_i$ and $Q_i$, respectively, which belongs to $T_1$. Let $a_i$ and $b_i$ be the embedding in $V_1$ of $p_i$ and $q_i$, respectively, and let $A=\{a_i,b_i:i \in [\ell]\}$. 

By \cref{lem:degree_condition_inherited_by_random_subsets} again, we have, with high probability, $\delta^0(D[A \cup V_2]) \geq (1/2 + \alpha/2)\size{A \cup V_2}$. Applying \cref{lem:CSaspaths} to $D[A \cup V_2]$ with $T_i = F_i \cup P_i \cup Q_i$, $r_i = p_i$, and $s_i = q_i$, for each $i\in [\ell]$, we can find a copy of $T_2$ in $D[V_1 \cup V_2]$. Now since $T_2$ is a tree, any vertex in $T_3 \setminus T_2$ can have at most one neighbour in $T_2$. Note that, by \cref{lem:degree_condition_inherited_by_random_subsets}, we know that with high probability every vertex in $D$ has at least $(1/2 + \alpha/2) \size{V_3}\geq \eta n$ in-neighbours in $V_3$ and at least $(1/2 + \alpha/2) \size{V_3} \geq \eta n$ out-neighbours in $V_3$. Let $j=|T_3|-|T_2|\leq \eta n$ and order the vertices of $T_3 \setminus T_2$ by $u_1, \dots, u_j$, so that $T[V(T_2) \cup \{u_1, \dots, u_i\}]$ is a tree for each $i\in [j]$. Embed the vertices $u_1,\ldots,u_j$ greedily into $V_3$, to complete the copy of $T$ in $D$. Noting that this embedding was successful with probability at least $\eps/3-o(1)>0$, there must always be such a copy of $T$.
\end{proof} 

%% file: 04_absorption.tex
\section{Absorption from switching}\label{sec:switching}

The aim of this section is to prove \cref{switching}. The main idea is as follows. Given a small tree $T$, we split it into two trees $T'$ and $T''$ and randomly embed $T'$ vertex by vertex. With positive probability, the resulting tree is such that, given the right number of other vertices in the graph, we can embed $T''$ to extend this into a copy of $T$ while making some small modifications to the copy of $T'$. Essentially, we show that, for each vertex $y$, there are many vertices in the embedding of $T'$ which we can switch with $y$ and still get a copy of $T$. We then embed $T''$ vertex-by-vertex, at each step switching an unused vertex into the copy of $T'$ in place of a vertex which we can instead use to extend the (partial) embedding of $T''$. 

\begin{proof}[Proof of \cref{switching}] Take $\lambda$ such that $\eps\ll\lambda \ll \mu$. Using Proposition~\ref{littletree}, let $T=T'\cup T''$, where $t\in V(T')$ and $\eps n< |T''|\leq 3\eps n$.
Let $\ell=|T'|$, and label $V(T')$ as $t_1,\ldots,t_\ell$ so that $t_1=t$, $T'[t_1,\ldots,t_i]$ is a tree for each $i\in [\ell]$, and the leaves of $T'$ appear last in this order (except for $t$) and in any bare path of length 6 the middle 3 vertices appear consecutively. For each $i\in [\ell]$, let $T_i=T'[\{t_1,\ldots,t_i\}]$.

Pick an arbitrary vertex $v\in V(D)$, and let $R_1$ be the graph with only the vertex $v$. For each $i=2,\dots,\ell$, do the following. Let $\diamond_i\in \{+,-\}$ be such that $N^{\diamond_i}_{T_{i}}(t_i)$ is non-empty (and thus contains exactly one vertex. Let $\circ_i \in \{+,-\}$ with $\circ_i \neq \diamond_i$. Take $R_{i-1}$, which is a copy of $T_{i-1}$, and let $w_i$ be the copy of the sole vertex in $N^{\diamond_i}_{T_{i}}(t_i)$ in $R_{i-1}$. 
Pick a vertex $v_i$ independently at random from $N_D^{\circ_i}(w_i)\setminus V(R_{i-1})$. Embed $t_i$ to $v_i$ to get $R_i$, a copy of $T_i$.

Note that this process always ends with a copy of $T'$, as $N_D^{\circ_i}(w_i)\setminus V(R_{i-1})$ always has size at least $d^{\circ_i}_D(w_i)-|T|\geq (1/2+\alpha)n-\mu n$ and $\mu\ll \alpha$. Let $R=R_\ell$, so that $R$ is a copy of $T'$. We will show that, with positive probability the following property holds.
\begin{enumerate}[label = \textbf{S}]
    \item For each distinct $x,y\in V(D)$ and $\diamond \in \{+,-\}$,\label{propP}
    \[
    |\{i\in [\ell]:v_i\in N^\diamond_D(x)\text{ and }N^\pm_R(v_i)\subset N^\pm_D(y)\}|\geq \lambda n.
    \]
\end{enumerate}

Noting $|R|=|T'|\leq |T|-|T''|+1\leq (\mu-\eps)n$, let $A\subset V(D)$ contain $V(R)$ so that $|A|=(\mu-\eps)n$, and let $v$ be the copy of $t$. We will show in two claims that, with positive probability \ref{propP} holds, and that, if \ref{propP} holds, then $A$ and $v$ satisfy the property in the theorem. Thus, the theorem follows from these two claims.

\begin{claim}\label{cl:P}
With positive probability, \ref{propP} holds.
\end{claim}
\begin{proof}[Proof of Claim~\ref{cl:P}] Fix $x,y\in V(D)$ and $\diamond \in \{+,-\}$ with $x\neq y$. We will show that \ref{propP} holds for $x,y$ and $\diamond$ with probability at least $1-1/4n^2$, so that the result follows by a union bound.

For convenience, let us take two cases. Either $T'$ has $2\mu^2 n$ leaves (Case I) or $\mu^2 n$ vertex-disjoint bare paths with length 6 (Case II). One of these cases must hold, as, suppose that Case I does not hold and thus $T'$ has fewer than $2\mu^2n$ leaves. Then, by \cref{lem:few_leaves_many_bare_paths}, we know that there is some $s$ and some vertex-disjoint bare paths $P_i$, $i \in [s]$, in $T'$ of length $6$ so that $\size{T'-P_1 - \dots -P_s} \leq 72 \mu^2 n + 2\ell/7$. Removing the internal vertices of each path $P_i$, $i\in [s]$, from $T'$ removes 5 vertices, and $|T'|=\ell$, so that $\ell-5s\leq 72\mu^2n+2\ell/7$, and therefore
\[
s \geq (\ell-2\ell/7)/5 -72 \mu^2 n/5 \geq \ell/7-15\mu^2n\geq (\mu-3\eps)n/7-15\mu^2 n\geq \mu^2 n,
\]
where the final inequality holds since $\eps \ll \mu$.

\smallskip

\textbf{Case I.} Assume that at least $\mu^2 n$ leaves of $T'$ are out-leaves, where the proof whenever $T'$ has at least $\mu^2n$ in-leaves follows similarly. Let $\ell'$ be the smallest integer such that, for each $i>\ell'$, $t_i$ is a leaf of $T'$. We will analyse the embedding of $T'$ in two stages. First, for the embedding of $t_1,\ldots,t_{\ell'}$, we show that with high probability there will be plenty of these vertices which are adjacent to out-leaves in $t_{\ell'+1},\ldots,t_\ell$ that are embedded to in-neighbours of $y$. Then, we will analyse the embedding of $t_{\ell'+1},\ldots,t_\ell$, and show that plenty of these vertices whose in-neighbour in $t_{1},\ldots,t_{\ell'}$ was embedded to an in-neighbour of $y$ are themselves embedded to a $\diamond$-neighbour of $x$. Here, we require that the leaves we consider are at distance at least 2 from $t$ in $T'$. Let $\mu' = \mu/2$. By the degree condition, at least $\mu' n$ of these out-leaves are at distance at least 2 from $t$.

For each $i\in [\ell']$, let $c_i$ be the number of out-leaves of $t_i$ in $T'$. For each $i\in [\ell']$, let $X_i$ be the random variable which takes value $c_i$ if $v_i \in N^-_D(y)$, and 0 otherwise. Note that, for each $i\in [\ell]$, if $c_i>0$, then, when the process selects $v_i$, having chosen $v_1,\ldots,v_{i-1}$, $X_i=c_i$ with probability at least
\begin{equation}\label{thatonethere} 
\frac{|(N^{\circ_i}_D(w_i)\setminus V(R_i))\cap N^-_D(y)|}{n}\geq \frac{|(N^{\circ_i}_D(w_i))\cap N^-_D(y)|-|R_i|}{n}\geq \frac{2\alpha n-\mu' n}{n}\geq \alpha,
\end{equation}
as $\alpha\gg\mu$. Thus, for each $i\in [\ell]$, $\E[X_{i} \mid X_1, \dots\ X_{i-1}]\geq \alpha c_{i}$.

Note that $\sum_{i\in [\ell']}c_i$ is the number of out-leaves of $T'$, so that $\sum_{i\in [\ell']}c_i\geq (\mu')^2n$. On the other hand, clearly $\sum_{i\in [\ell']}c_i \leq n$, and for every $i \in [\ell']$, $c_I \leq \Delta(T)\leq cn/\log n$. Thus, as before by repeated application of \eqref{eq:quadratic_inequality}, we have $\sum_{i\in [\ell']}c_i^2\leq cn^2/\log n$. Note that $\size{X_i-\alpha c_i} \leq c_i$ for each $i\in [\ell']$. Therefore, we can apply \cref{cor:azuma}~\ref{cor:azuma_sub} with $t=\alpha (\mu')^2 n/2$ to get
%
\begin{equation}\label{eq1}
\Prob\lsb \sum_{i\in [\ell']}X_i \leq \sum_{i\in [\ell']}\alpha c_i-t \rsb \leq 2 \exp \lb \dfrac{-t^2}{\sum_{i\in [\ell']}c_i^2} \rb \leq 2 \exp \lb \dfrac{-t^2 \log n}{cn^2} \rb \leq \frac{1}{8n^2}.
\end{equation}
Here, the final inequality holds because $c \ll \mu,\alpha$. Therefore, with probability at least $1-1/8n^2$, we have
$\sum_{i\in [\ell']} X_i \geq \sum_{i\in [\ell']}\alpha c_i-\alpha (\mu')^2 n/2\geq \alpha (\mu')^2 n/2$.

Let $m=\sum_{i\in [\ell']}X_i\geq \alpha (\mu')^2 n/2$. Consider now the embedding of  $t_{\ell'+1},\ldots,t_{\ell}$. Let $j_1,\ldots,j_m\in \{\ell'+1,\ldots,\ell\}$ be such that $t_{j_i}$ is an out-leaf of $T'$ and the image of $N^{-}_{T'}(t_{j_i})$ is an in-neighbour of $y$ for each $i \in [m]$. For each $i \in [m]$, let $Y_i$ be the random variable which takes value $1$ if $v_{j_i}$ is in $N^\diamond_D(x)$, and 0 otherwise. Note that, similarly to the calculation in \eqref{thatonethere}, $\E[Y_i\mid Y_1,\ldots Y_{i-1}]\geq \alpha$ for each $i\in [m]$. As before, since $\size{Y_i-\alpha } \leq 1 - \alpha$ for each $i\in [m]$ as $\alpha\leq 1$, we can apply \cref{cor:azuma}~\ref{cor:azuma_sub} with $t = \alpha m/2$ to get
\begin{equation}\label{eq2}
\Prob\lsb \sum_{i\in [m]} Y_i < \alpha m -t \rsb \leq 2\exp \lb \frac{-t^2}{(1-\alpha)^2 m} \rb \leq \frac{1}{8n^2},
\end{equation}
where the final inequality holds because $1/n\ll \mu, \alpha$. Hence, with probability at least $1-1/8n^2$, we have $\sum_{i\in [m]}Y_i\geq \alpha m/2$. Note that $\size{\{i\in [\ell]:v_i\in N^\diamond_D(x)\text{ and }N^\pm_R(v_i)\subset N^\pm_D(y)\}}\geq \sum_iY_i$. 

Thus, by taking a simple union bound over the events in \eqref{eq1} and \eqref{eq2} and using $\lambda\ll \alpha,\mu$, we see that in total, with probability at least $1-1/4n^2$, 
\[
\size{\{i\in [\ell] \colon v_i\in N^\diamond_D(x)\text{ and }N^\pm_R(v_i)\subset N^\pm_D(y)\}}\geq \alpha m/2\geq \lambda n.
\]
Taking a union bound over all possible $x, y \in V(D)$ and $\diamond \in \{ +, -\}$, we see that in this case \ref{propP} holds with probability at least $1/2$. 
\smallskip

\textbf{Case II.} Let $m=\mu^2 n$. Let $P_1,\ldots,P_m$ be vertex disjoint paths of length 6 in $T$, so that, if, for each $i\in [m]$, $j_i$ is such that $t_{j_i}$ is the middle vertex of $P_i$, then the vertices $t_{j_i}$ appear in order in $t_1,\ldots,t_\ell$.

For each $i\in [m]$, let $X_i$ be the random variable taking value 1 if 
\begin{equation}\label{thisonehere}
v_{j_i} \in N^\diamond_D(x)\text{ and }N^\pm_R(v_{j_i})\subset N^\pm_D(y)
\end{equation}
and 0 otherwise. Note that, by virtue of the labelling of the $t_1,\ldots,t_\ell$, the vertices that appear in $N^\pm_R(v_{j_i})$ are exactly the vertices $v_{j_i-1}$ and $v_{j_i+1}$. When we choose each of $v_{j_i-1},v_{j_i},v_{j_i+1}$, the probability that it satisfies its condition in \eqref{thisonehere} (however the previous vertices $v_{i'}$ are chosen) is at least $\alpha$, in a calculation similar to \eqref{thatonethere}. Therefore, we have, for each $i\in [m]$, that  $\E[X_{i} \mid X_1, \dots\ X_{i-1}]\geq \alpha^3$. Since $\size{X_i-\alpha^3} \leq 1$ for each $i\in [m]$ as $\alpha\leq 1$, we can apply \cref{cor:azuma}~\ref{cor:azuma_sub} with $t=\alpha^3m/2$ to get 
\[
\Prob\left[\sum_{i\in [m]} X_i \leq \alpha^3m-t\right] \leq 2 \exp \lb \frac{-t^2}{m} \rb = 2 \exp \lb \frac{-\alpha^6m}{4} \rb \leq \frac{1}{4n^2},
\]
as $1/n\ll \alpha,\mu$.
Therefore, with probability at least $1-1/4n^2$, as $\lambda\ll \mu,\alpha$,
\begin{align*}
\size{\{i\in [\ell] \colon v_i\in N^\diamond_D(x)\text{ and }N^\pm_R(v_i)\subset N^\pm_D(y)\}}
&\geq \size{\{i\in [m]:v_{j_i}\in N^\diamond_D(x)\text{ and }N^\pm_R(v_{j_i})\subset N^\pm_D(y)\}}\\
&= \sum_{i\in [m]} X_i\geq \alpha^3 m/2\geq  \lambda n.
\end{align*}
Taking a union bound over all possible $x, y \in V(D)$ and $\diamond \in \{ +, -\}$, we see that in this case \ref{propP} holds with probability at least $1/2$.
\end{proof}

\begin{claim}\label{finalclaim}
If \ref{propP} holds then $A$ and $v$ satisfy the property in the theorem.
\end{claim}

\medskip

\noindent\textit{Proof of Claim~\ref{finalclaim}.}
Let $B\subset V(D)$ with $A\subset B$ and $|B|=\mu n$.
Let $k=|T''|-1\leq 3\eps n$ and label the vertices of $V(T'')\setminus V(T')$ as $s_1,\ldots,s_k$, so that, for each $i\in [k]$, $T'_i:=T'\cup T''[\{s_1,\ldots,s_i\}]$ is a tree. Note that $|B\setminus V(R)|=k$ and label the vertices of $B\setminus V(R)$ as $y_1,\ldots,y_k$. 

Let $S_0=R$. Now, for each $i=1,\ldots,k$ in turn, do the following. Let $x_i\in V(S_{i-1})$ and $\diamond_i\in \{+,-\}$ be such that we need to add a $\diamond_i$-neighbour to $x_i$ as a leaf to get a copy of $T'_i$. 
Choose some $j'_i\in [\ell]\setminus \{1,j'_1,\ldots,j'_{i-1}\}$ such that 
\[
v_{j'_i}\in N^{\diamond_i}_D(x_i)\;\text{ and }\;N^\pm_{S_{i-1}}(v_{j'_{i}})\subset N^\pm_D(y_i)\;\text{ and }\; d^+_{S_{i-1}}(v_{j'_i})+d^-_{S_{i-1}}(v_{j'_i})\leq 4/\lambda.
\]
Replace $v_{j'_i}$ with $y_i$ in $S_{i-1}$ and add $v_{j'_i}$ as a $\diamond_i$-neighbour of $x_i$ to get $S_i$, a copy of $T'_i$ with vertex sets $V(S_{i-1})\cup \{y_i\}$.

We need only show that there is such a vertex $v_{j'_i}$ in each case, as if this process finds $S_k$, then we have a copy of $T_k'=T$. Fix then $i\in [k]$. By \ref{propP}, we know there are at least $\lambda n$ choices of $i' \in [\ell]$ such that $v_{i'} \in N^{\diamond_i}_D(x_i)$ and $N^\pm_R(v_{i'})\subset N^\pm_D(y_{i})$. By the construction of $S_{i-1}$, there are at most $(4/\lambda)\cdot 3\eps n\leq \lambda n/4$ vertices adjacent to the vertices $v_{j'_1},\ldots,v_{j'_{i-1}}$ in $S_{i-1}$, and so at most $\lambda n/4$ vertices of $R$ can be adjacent to the vertices $v_{j'_1},\ldots,v_{j'_{i-1}}$ in $S_{i-1}$. Therefore for all but at most $\lambda n/4$ values of $i'\in [\ell]$, we have $N^+_R(v_{i'})=N^+_{S_{i-1}}(v_{i'})$ and $N^-_R(v_{i'})=N^-_{S_{i-1}}(v_{i'})$. Furthermore, as $\sum_{i'\in [\ell]}(d^+_{T}(t_{i'})+d^-_T(t_{i'}))\leq 2n$, at most $\lambda n/2$ values of $i\in [k]$ can have  $d^+_{S_{i-1}}(v_{j'_i})+d^-_{S_{i-1}}(v_{j'_i})> 4/\lambda$. Since $k \ll \lambda$, we know that there will be at least $\lambda n/8$ choices for $j'_i$, and so such a $j'_i$ will always exist by \ref{propP}. \hspace{11cm} \qed
\end{proof}

%% file: 00_main.bbl
\begin{thebibliography}{10}

\bibitem{alon2004probabilistic}
N.~Alon and J.~H. Spencer.
\newblock {\em \emph{\textbf{The {P}robabilistic {M}ethod}}}.
\newblock John Wiley \& Sons, 2004.

\bibitem{belabook}
B.~Bollob{\'a}s.
\newblock {\em \textbf{\emph{Extremal {G}raph {T}heory}}}.
\newblock Courier Corporation, 2004.

\bibitem{bottsurvey}
J.~B{\"o}ttcher.
\newblock Large-scale structures in random graphs.
\newblock {\em Surveys in Combinatorics}, 440(2017):87, 2017.

\bibitem{bandwidth}
J.~B{\"o}ttcher, M.~Schacht, and A.~Taraz.
\newblock Proof of the bandwidth conjecture of {B}ollob{\'a}s and {K}oml{\'o}s.
\newblock {\em Mathematische Annalen}, 343(1):175--205, 2009.

\bibitem{csaba2010tight}
B.~Csaba, J.~Nagy-Gy{\"o}rgy, I.~Levitt, and E.~Szemer{\'e}di.
\newblock Tight bounds for embedding bounded degree trees.
\newblock In {\em Fete of combinatorics and computer science}, pages 95--137.
  Springer, 2010.

\bibitem{anyori}
L.~DeBiasio, D.~K{\"u}hn, T.~Molla, D.~Osthus, and A.~Taylor.
\newblock Arbitrary orientations of {H}amilton cycles in digraphs.
\newblock {\em SIAM Journal on Discrete Mathematics}, 29(3):1553--1584, 2015.

\bibitem{DeMol}
L.~DeBiasio and T.~Molla.
\newblock Semi-degree threshold for anti-directed {H}amiltonian cycles.
\newblock {\em The Electronic Journal of Combinatorics}, P4.34, 2015.

\bibitem{GH}
A.~Ghouila-Houri.
\newblock Une condition suffisante d'existence d'un circuit {H}amiltonien.
\newblock {\em C.R.\ Acad. Sci. Paris}, 251(4):495--497, 1960.

\bibitem{HajSzem}
A.~Hajnal and E.~Szemer{\'e}di.
\newblock Proof of a conjecture of {P}.\ {E}rd{\H{o}}s.
\newblock {\em Combinatorial theory and its applications}, 2:601--623, 1970.

\bibitem{janson2011random}
S.~Janson, T.~{\L{}}uczak, and A.~Ruci{\'n}ski.
\newblock {\em {\emph{\textbf{Random {G}raphs}}}}.
\newblock John Wiley \& Sons, 2011.

\bibitem{KSS95}
J.~Koml{\'o}s, G.~N. S{\'a}rk{\"o}zy, and E.~Szemer{\'e}di.
\newblock Proof of a packing conjecture of {B}ollob{\'a}s.
\newblock {\em Combinatorics, Probability and Computing}, 4(3):241--255, 1995.

\bibitem{KSSpowercycle}
J.~Koml{\'o}s, G.~N. S{\'a}rk{\"o}zy, and E.~Szemer{\'e}di.
\newblock Proof of the {S}eymour conjecture for large graphs.
\newblock {\em Annals of Combinatorics}, 2(1):43--60, 1998.

\bibitem{KSS01}
J.~Koml{\'o}s, G.~N. S{\'a}rk{\"o}zy, and E.~Szemer{\'e}di.
\newblock Spanning trees in dense graphs.
\newblock {\em Combinatorics, Probability \& Computing}, 10(5):397, 2001.

\bibitem{DDgraphsurvey}
D.~K{\"u}hn and D.~Osthus.
\newblock Embedding large subgraphs into dense graphs.
\newblock {\em Surveys in Combinatorics}, pages 137--167, 2009.

\bibitem{DDminfactor}
D.~K{\"u}hn and D.~Osthus.
\newblock The minimum degree threshold for perfect graph packings.
\newblock {\em Combinatorica}, 29(1):65--107, 2009.

\bibitem{montgomery2019}
R.~Montgomery.
\newblock Spanning trees in random graphs.
\newblock {\em Advances in Mathematics}, 356:106793, 2019.

\bibitem{montgomery2018embedding}
R.~Montgomery, A.~Pokrovskiy, and B.~Sudakov.
\newblock Embedding rainbow trees with applications to graph labelling and
  decomposition.
\newblock {\em Journal of the European Mathematical Society},
  22(10):3101--3132, 2020.

\bibitem{RichardTassio}
R.~Mycroft and T.~Naia.
\newblock Spanning trees of dense directed graphs.
\newblock {\em Electronic Notes in Theoretical Computer Science}, 346:645--654,
  2019.

\bibitem{mycroft2020trees}
R.~Mycroft and T.~Naia.
\newblock Trees and tree-like structures in dense digraphs, 2020. \emph{arXiv
  preprint:} 2012.09201.

\bibitem{RRSab}
V.~R{\"o}dl, E.~Szemer{\'e}di, and A.~Ruci{\'n}ski.
\newblock An approximate {D}irac-type theorem for {$k$}-uniform hypergraphs.
\newblock {\em Combinatorica}, 28(2):229--260, 2008.

\end{thebibliography}
